\documentclass[11pt]{amsart}

 \usepackage{amssymb,amsmath,amsfonts,amsthm,array,comment,mathrsfs,times}

\makeindex
\usepackage{geometry}
\usepackage{longtable,tabularx}
 
\usepackage[all,cmtip]{xy}
\usepackage{tikz-cd}
\usepackage{url}
\usepackage[normalem]{ulem}
\usepackage[hidelinks]{hyperref}
\usepackage{scalerel}

\hypersetup{bookmarksdepth=2}

\usepackage{enumitem}
\setlist[enumerate]{label=(\roman*), font=\upshape}

\usepackage{graphicx}
\setlist[itemize,1]{
    label=\raisebox{0.15ex}{\scalebox{0.75}{$\bullet$}}
}

\newcommand{\Z}{\mathbb{Z}}

\newcommand{\Q}{\mathbb{Q}}

\newcommand{\Gal}{\mathrm{Gal}}

\newcommand{\F}{\mathbb{F}}

\def\mf{\mathfrak}
\def\mb{\mathbb}
\def\mc{\mathcal}
\def\mr{\mathrm}

\makeatletter
\@namedef{subjclassname@2010}{%
  \textup{2010} Mathematics Subject Classification}
\makeatother

\usepackage[T1]{fontenc}

\newtheorem{theorem}{Theorem}[section]
\newtheorem{corollary}[theorem]{Corollary}
\newtheorem{lemma}[theorem]{Lemma}
\newtheorem{proposition}[theorem]{Proposition}

\newtheorem{question}[theorem]{Question}

\theoremstyle{definition}
\newtheorem{definition}[theorem]{Definition}
\newtheorem{remark}[theorem]{Remark}
\newtheorem{example}[theorem]{Example}
\newtheorem{fact}[theorem]{Fact}
\newtheorem{notation}[theorem]{Notation}

\newtheorem*{mainthm A}{Main Theorem A}
\newtheorem*{mainthm B}{Main Theorem B}
\newtheorem*{mainthm C}{Main Theorem C}
\newtheorem*{mainthm D}{Main Theorem D}

\numberwithin{equation}{section}

\def\res{\operatorname{res}}
\def\tp{\operatorname{tp}}
\def\Th{\operatorname{Th}}
\def\mh{\mbox{-}}

\begin{document}\label{sec: introduction}

\title{On the structure of certain valued fields II}

\author[J. Lee]{Junguk Lee}
\address{Department of Mathematics, Changwon National University, 20 Changwondaehakro, Uichang-gu, Changwon-si, Gyeongsangnam-do, 51140, South Korea}
\email{ljwhayo@changwon.ac.kr}

\author[W. Lee]{Wan Lee}
\address{Department of Mathematics, Changwon National University, 20 Changwondaehakro, Uichang-gu, Changwon-si, Gyeongsangnam-do, 51140, South Korea}
\email{wannim@gmail.com}

\date{}

\begin{abstract}
We study the structure of finitely ramified henselian valued fields of mixed characteristic with arbitrary residue fields via their residue rings of higher length, where a residue ring of length $n$ is the quotient ring of the valuation ring by the $n$-th power of its maximal ideal. We prove an approximate lifting theorem for homomorphisms between residue rings of higher length and explicitly determine the optimal bound of error of lifting.

As applications, we obtain several Ax-Kochen-Ershov principles for relative completeness, relative existential completeness, and existential closedness using the pure ring structures on the residue rings.
And we show that any formula on a residue ring of length $k$  is equivalent to a normal form of sentences in pure ordered group structure and special formulas at level $n$, uniformly for all finitely ramified henselian valued fields of mixed characteristic $(0,p)$ and the same initial ramification index $e$.
Here, for $n\ge k$, a special formula at level $n$ represents a definable set on a residue ring of length $k$ given by a projection image of a definable set in pure ring structure on a residue ring of length $n$. 
Also, such $n$ is optimally computed from the precise estimation of error of lifting and depends only on $k$, residue characteristic $p$, and initial ramification index $e$.
\end{abstract}



\maketitle
\tableofcontents

\section{Introduction}
We will study the structure of finitely ramified mixed characteristic valued fields via their residue rings. A residue ring of length $n$ is a quotient ring of the valuation ring by the $n$-th power of its maximal ideal. We will mainly consider the following question.

\begin{question}\label{question:motivation}
Fix a prime $p$ and a positive integer $e$. What is the lifting number $N_p(e)$ : the smallest $n$ such that given two finitely ramified complete discretely valued fields $K_1$ and $K_2$ of mixed characteristic $(0,p)$ and the initial ramification index $e$, $K_1$ and $K_2$ are isomorphic if two residue rings of length $n$ are isomorphic\textup{?}
\end{question}

\noindent In \cite{LL21}, we gave an upper bound of $N_p(e)$ for the perfect residue fields and in \cite{ADJ24}, Anscombe, Dittmann, and Jahnke gave an upper bound of $N_p(e)$ for arbitrary residue fields. Recently, Velavan in \cite{Vel26} gave an optimal bound of $n$ by considering a quotient ring of the valuation ring by the ideal generated by $p^n$ in Question \ref{question:motivation}.\\

In this paper, we explicitly determine the lifting number $N_p(e)$. Hou and Keating~\cite[Proposition~2.1]{HK04} obtained an upper bound for $N_p(e)$ for finite residue fields. We allow arbitrary residue fields and prove that this bound is optimal when $p\mid e$.
\begin{mainthm A}[Theorem \ref{thm: lifting number}]
For every prime $p$ and positive integer $e$, lifting numbers are given as follows:
$$
N_{p} (e)=
\begin{cases}
1,&  \hbox{ if }\ e=1,\\
e+1,&  \hbox{ if } \ e>1\ \hbox{ and }\  p\nmid e,\\
e(1+v_p(e))+\lfloor \dfrac{e}{p-1}\rfloor+1,&  \hbox{ if } \  p\mid e,
\end{cases}
$$
where $\left\lfloor \cdot \right\rfloor$ denotes the greatest integer function.
\end{mainthm A}
\noindent In addition, we prove that for $n$ large enough, any homomorphism between residue rings of length $n$ can be approximately lifted into an embedding of valued fields with an explicit error bound.

\begin{mainthm B}[Theorem \ref{thm: approximate lifting} and the proof of Theorem~\ref{thm:exact-normal-form-depth}]
Fix a prime $p$, a positive integer $e$, and  a positive integer $n\ge N_p(e)$. For any   complete discretely valued fields $K$ and $L$ of mixed characteristic $(0,p)$ and the initial ramification index $e$, if there is a homomorphism $\varphi$ between residue rings of length $n$, then there is an embedding $\Phi:K\to L$ such that $\Phi$ and $\varphi$ induce the same homomorphism on the residue ring of length $k$ for $1\le k\le n-(e-1+e v_p(e))$. If $\varphi$ is an isomorphism, $\Phi$ is an isomorphism. Also, the upper bound $n-(e-1+e v_p(e))$ is optimal.
\end{mainthm B}

As applications, we prove full and existential Ax-Kochen-Ershov principles, together with relative embedding and existential closedness results. 
And we show that any formula on a residue ring of length $k$  is equivalent to a normal form of sentences in pure ordered group structure and special formulas at level $n$, uniformly for all finitely ramified henselian valued fields of mixed characteristic $(0,p)$ and the same initial ramification index $e$. 
Such $n$ is also optimally computed. 
As a corollary, we deduce that the multi-sorted structure of residue rings of all finite lengths, connected by natural projection maps, is stably embedded in its pure ring structure for each residue ring  and is orthogonal to the pure value group (see Remark \ref{rem:meaning of induced structure}).

\begin{mainthm C}[Corollary \ref{cor:relative complete}, Corollary \ref{cor: relative existential}, and Theorem \ref{theorem:relative existential closedness}]
For a prime $p$ and a positive integer $e$, let $T_{p,e}$ be the $\mathcal L_{val}$-theory of all henselian valued fields of mixed characteristic $(0,p)$ and the initial ramification index $e$. 
Then, $N_{p}(e)$ is the smallest integer $n$ such that for any valued fields $(K,v)$ and $(L,w)$ satisfying $T_{p,e}$,
$$\underbrace{(K,v)\equiv (L,w)}_{\mbox{ in }\mathcal L_{val}}\ \Longleftrightarrow \  \underbrace{\mc O_{v,n}\equiv \mc O_{w,n}}_{\mbox{ in }\mathcal L_{ring}}\mbox{ and }\underbrace{v K\equiv w L}_{\mbox{ in }\mathcal L_{oag}}$$ and
$$\underbrace{(K,v)\equiv_{\exists} (L,w)}_{\mbox{ in }\mathcal L_{val}} \ \Longleftrightarrow \  \underbrace{\mc O_{v,n}\equiv_{\exists^+} \mc O_{w,n}}_{\mbox{ in }\mathcal L_{ring}}$$ and,
in case  $(K,v)\subseteq (L,w)$,
$$\underbrace{(K,v)\preceq_{\exists} (L,w)}_{\mbox{ in }\mathcal L_{val}} \ \Longleftrightarrow  \ \underbrace{\mc O_{v,n}\preceq_{\exists} \mc O_{w,n}}_{\mbox{ in }\mathcal L_{ring}}\mbox{ and }\underbrace{v K\preceq_{\exists} w L}_{\mbox{ in }\mathcal L_{oag}},$$
where $\mathcal L_{val}$ is the valued field language, $\mathcal L_{ring}$ is the ring language, and $\mathcal L_{oag}$ is the ordered group language. Here, $\equiv_{\exists^+}$ means satisfying the same positive existential sentences (cf. \cite[Definition~2.1.1 and Remark~2.1.3]{Kam25}). 
\end{mainthm C}

\begin{mainthm D}[Proposition \ref{prop:equivlance between formula and special formula} and Theorem \ref{thm:exact-normal-form-depth}]
For a prime $p$ and positive integers $e$ and $k$, let $T_{p,e}$ be the $\mathcal L_{val}$-theory of all henselian valued fields of mixed characteristic $(0,p)$ and the initial ramification index $e$, and let $\lambda_{p}(e,k):=\mr {max}\{k+e-1+ev_p(e), N_p(e) \}$. 
Then, $\lambda_p(e,k)$ is the smallest integer $n$ with $n\ge k$ such that for a tuple $x$  of variables on the residue ring of length $k$, any $\mathcal L_{val}$-formula $\theta(x)$ is equivalent modulo $T_{p,e}$ to a formula of the form 
$$\bigvee\left( \theta_{0}(x)\wedge \theta_{1}\right),$$
where  $\theta_0(x)$ is a special formula at level $n$, that is,
$$\theta_0(x)\equiv \exists x'\left(\theta'(x')\wedge \res^n_k(x')=x\right),$$ where $x'$ is a tuple of variables on the residue ring of length $n$ and $\theta'(x')$ is a $\mathcal L_{ring}$-formula, and $\theta_1$ is a $\mathcal L_{oag}$-sentence. Moreover, if $\theta(x)$ is existential, then it is equivalent modulo $T_{p,e}$ to a single special formula $\theta_0(x)$ at level $n$ where $\theta'(x')$ is a positive existential $\mathcal L_{ring}$-formula.
\end{mainthm D}

\medskip

The paper is organized as follows.
Section~\ref{sec: preliminaries} develops lifting results using formal smoothness, constructs Cohen subrings, and proves root estimates.
Section~\ref{sec:embedding lemma} proves an approximate lifting theorem and computes the lifting number.
Section~\ref{sect:AKE via higher residue rings} proves the full and existential AKE principles and their decidability consequences, together with relative embedding and existential closedness results.
Section~\ref{sec:induced} proves stable embeddedness and orthogonality, obtains special normal forms, and computes their least uniform level of the normal forms. 
It also shows that these normal forms are effectively computable.

\section{Preliminaries}\label{sec: preliminaries}
\subsection{Notation}
Given a valued field $(K,v)$, we denote the residue field by $Kv$, the value group by $vK$, the valuation ring by $\mathcal O_v$ and the maximal ideal by $\mathfrak{m}_v$.
When there is no confusion on $v$, we write $\mc O_K$ for $\mc O_v$.
For $n\ge 1$, define
$$
\mathcal O_{v,n}=\mathcal O_v/\mathfrak{m}_v^n \ \hbox{ and }\  \res_n:\mathcal O_v\longrightarrow \mathcal O_{v,n},\quad x\longmapsto x+\mathfrak{m}_v^n
$$
as the residue ring of length $n$  and the residue map at level $n$, respectively.
 For $m \ge n\ge 1$, let 
 $$
 \res^m_n:\mathcal O_{v,m}\longrightarrow \mathcal O_{v,n},  \quad x+\mathfrak{m}_v^m\longmapsto x+\mathfrak{m}_v^n
 $$ 
 be the natural projection. 
 For $x\in\mathcal O_v$, we simply write $\bar{x}:=\operatorname{res}_1(x)\in Kv$. 
 Similarly, for $h(X)\in\mathcal O_v[X]$, we denote its coefficientwise reduction by $\bar{h}\in Kv[X]$.
 All ring homomorphisms are assumed to be unital.

 Let $(K,v)$ be a valued field of mixed characteristic $(0,p)$. 
 We call $(K,v)$ \emph{finitely ramified} if the set
$$
\{\gamma\in vK:0_{vK}<\gamma\le v(p)\}
$$
is finite, and call its cardinality $e_K$ the \emph{initial ramification index} of $(K,v)$. 
In this case, $vK$ has a least positive element $\epsilon$, and $v(p)=e_K\epsilon$. 
For a discretely valued field, $e_K$ agrees with the usual absolute ramification index; in particular, $e_K=v(p)$ when the valuation is normalized at a uniformizer.

 For a $\mb Z$-valued field we usually  normalize $\epsilon=1$  unless stated otherwise.
 When we restrict a valuation to a subfield, we keep the same normalization unless stated otherwise.
We use the usual conventions
$
v(0)=\infty \ \ \text{and} \ \ \infty+\gamma=\infty.
$

 We say a finite extension $(F,v_0)\subseteq (K,v)$ of valued fields is {\em totally ramified} if $$[K:F]=[v K : v_0 F],$$ and so $Kv=Fv_0$.

In this paper, we consider finitely ramified valued fields as first order structures in the following multi-sorted language $\mathcal L_{val}$: there are a sort $K$ for the field, a sort $\Gamma$ for the value group with infinity, and a sort $\mathcal O_n$ for the residue ring of length $n$ for each $n\ge 1$. 
The ring sorts $K$ and $\mathcal O_n$ are each endowed with the ring language $\mathcal L_{ring}:=\{+,\cdot,0,1\}$ and the value group sort with $\mathcal L_{oag}:=\{+,<,0,\infty\}$. 
Also, there are symbols for the valuation map $v:K\twoheadrightarrow \Gamma\cup\{\infty\}$, the residue map at level $n$ $\res_n:K\twoheadrightarrow \mathcal O_n$, which we interpret as the constant zero map outside the valuation ring, and the natural projection $\res^m_n:\mathcal O_m\twoheadrightarrow \mathcal O_n$ for $m > n\ge 1$, with $\mr{res}_n^n=\mr{id}_{\mathcal O_n}$.
We use $\operatorname{res}_k^k(x)$ as an abbreviation for $x$. We write $T_{p, e}$ for the theory of all finitely ramified henselian fields of mixed characteristic $(0,p)$ with $e_K=e$, in this language.

  We call an embedding $\Phi:(K,v)\hookrightarrow (L,w)$ isometric if for $x,y\in K $
$$
v(x)>v(y)\quad\Longleftrightarrow \quad w(\Phi(x))>w(\Phi(y)) .
$$
For   discretely valued fields, this is equivalent to continuity. For finitely ramified valued fields of the same mixed characteristic and initial ramification index, every isometric embedding induces a unique \(\mathcal L_{\mathrm{val}}\)-embedding.

\subsection{Formal smoothness}

We recall some lifting properties from algebraic geometry and commutative ring theory.
\begin{definition}
    Let $A\rightarrow B$ be a ring homomorphism and $I\subset B$ an ideal.
 The map  $A\rightarrow B$ is   \emph{formally smooth  for the $I$-adic topology} if,  for any $A$-algebra $C$,   ideal $N \subset C$ with $N^2=0$,   and   $A$-algebra homomorphism $u:B\rightarrow C/N$ with $u(I^n)=0$ for some   $n$,  there is an $A$-algebra lifting  $t:B\rightarrow C$ of $u$, i.e., $t$ makes the following diagram
     \begin{equation*}
    \begin{tikzcd}
        B \arrow[r, "u"]  \arrow[rd, dotted, "t" ] & C/N    \\
        A \arrow[r] \arrow[u]  & C   \arrow[u] 
    \end{tikzcd}
\end{equation*} 
 commutative.
  It is  \emph{formally unramified}   if there is at most one such lifting, and \emph{formally \'{e}tale}  if there is exactly one such lifting.
 When $I=0$, we omit the reference to the  topology.
\end{definition}
\noindent
Note that    $u(I^n)=0$ is equivalent to saying that $u$ is continuous for the discrete topology on $C/N$.
 Since $N^2=0$, we   have   $t(I^{2n})=0$, i.e., $t$ is continuous for the discrete topology on $C$.

\begin{fact}[{\cite[Section 28]{Mat86}}]\label{lemma base change smooth}
  \begin{enumerate}[ font=\normalfont]
      \item 
    Suppose $A\rightarrow B$ is formally smooth (resp., formally unramified)   for the $I$-adic topology.
    Then for any $A$-algebra $A'$, the induced map $A'\rightarrow B'$ is formally smooth (resp., formally unramified)   for the $I'$-adic topology, where $B':=B\otimes_A A'$ and $I':=IB'$.
   \item 
    Let $A\stackrel{g}\rightarrow B \stackrel{g'}\rightarrow B'$ be  ring homomorphisms.
    Let $I$ and $I'$ be ideals of $B$ and $B'$ respectively.
    Suppose $g'(I^m)\subset (I')^n$ for some $m$ and $ n$.
    If $g$ is formally smooth (resp., formally unramified) for the $I$-adic topology and    $g'$ is formally smooth (resp., formally unramified)  for the $I'$-adic topology respectively, then $g'g$ is formally smooth (resp., formally unramified) for the $I'$-adic topology.
    \end{enumerate}
\end{fact}
 
\begin{definition}
    Let $l/k$ be a field extension.
   We say that  a subset $\underline{b}\subset l$  is  a \emph{differential basis} if it is a $p$-basis or a transcendence basis of $l/k$, depending on whether  $\hbox{char } l=p> 0$ or  $\hbox{char } l=0$, respectively.
\end{definition}

\begin{remark}[{\cite[Section~26]{Mat86}}]
  $\underline{b}$  is  a differential basis of $l/k$ if and only if the   image of  $\underline{b}$ under the natural map $l\rightarrow \Omega_{l/k}$ is   a basis of $\Omega_{l/k}$ over $l$, where $\Omega_{l/k}$ denotes the module of differentials of $l/k$.
\end{remark}

\begin{fact}[{\cite[Section 26]{Mat86}}] \label{lemma separable smooth criterion}
Let $l/k$ be a field extension.
\begin{enumerate}[ font=\normalfont]
\item 
     $l/k $ is formally smooth  
   if and only if
    $l/k$ is separable. 
    \item 
    Suppose that  $l/k$ is separable.
    Let $\underline{\beta} $ be a differential basis of $l/k$.
    Then $\underline{\beta} $ is algebraically independent over $k$ and $l/k(\underline{\beta})$ is formally \'{e}tale.    
    \end{enumerate}
\end{fact}

\begin{lemma}\label{lemma smooth lifting}
    Given     $A\rightarrow B$, let  $I\subset A$ be  an ideal with $I^2=0$.
    \begin{enumerate}[  font=\normalfont]
        \item 
       If $ A\rightarrow B$ is  flat and  $A/I\rightarrow B/IB$  is formally smooth, then $A\rightarrow B$ is formally smooth.
        \item 
       If $A/I\rightarrow B/IB$  is formally unramified, then  $ A\rightarrow  B$ is formally unramified.
    \end{enumerate}
    \end{lemma}
    
\begin{proof}
   For (i),  see \cite[Section 10.138, Tag 031L]{Stacks}.
   For  (ii),  note that $ A\rightarrow B$  is  formally unramified  if and only if  $\Omega_{B/A}=0$
   (see \cite[Section 25]{Mat86}).
   By  the base change, we have 
    $$
    \Omega_{B/A}/I\Omega_{B/A}= \Omega_{B/A} \otimes_B B/IB= \Omega_{B/A} \otimes_A A/I =   \Omega_{\left(B/IB\right)/\left(A/I\right)}=0.
    $$
         Thus $ \Omega_{B/A}  =  I\Omega_{B/A}=  I^2\Omega_{B/A} =0$.
  \end{proof}


\begin{lemma}\label{lemma DVR smooth etale lift}
    Let $ (A, \mathfrak{m}_A, k) \subset (B, \mathfrak{m}_B, l)$ be  DVRs with $\mathfrak{m}_AB=\mathfrak{m}_B^e$ for some $e\ge 1 $. 
        The following are equivalent:
    \begin{enumerate}[label=(\roman*), font=\normalfont]
        \item 
      $A\hookrightarrow B$ is formally  smooth (resp., formally \'{e}tale) for the $\mathfrak m_B$-adic topology.
        \item 
      $A/\mathfrak{m}_A^n\rightarrow B/\mathfrak{m}_A^nB$ is formally  smooth (resp., formally \'{e}tale) for all $n
      \ge 1$.
        \item
         $A/\mathfrak{m}_A^{n_0}\rightarrow B/\mathfrak{m}_A^{n_0}B$ is formally  smooth (resp., formally \'{e}tale) for some $n_0 \ge 1$.
        \item 
       The field extension $l/k$ is formally smooth (resp., formally \'{e}tale) and $e=1$.
    \end{enumerate}

\end{lemma}
\begin{proof}
 (i) $\Rightarrow$ (ii): this follows from Fact~\ref{lemma base change smooth} (i) since $\mathfrak{m}_AB$-adic and  $\mathfrak{m}_B$-adic topologies are the same. 
 
(ii) $\Rightarrow$ (iii):  obvious.

 (iii) $\Rightarrow$ (iv):
    by Fact~\ref{lemma base change smooth} (i),  the map  $k\rightarrow B/\mathfrak{m}_AB  =B/ \mf m _B^{e}$ is formally smooth (resp., formally \'{e}tale).
      Let $f$ be the prime field in $k$.
   Since $f$ is perfect, both $k/f$ and $l/f$ are  formally smooth by Fact~\ref{lemma separable smooth criterion} (i).
        Thus the composite map $f\rightarrow k \rightarrow B/\mathfrak{m}_A B$ is formally smooth by Fact~\ref{lemma base change smooth} (ii).
The identity map $l \xrightarrow{=} B/\mf m _B$ lifts successively through the quotients $B/\mf m_B^{i+1} \rightarrow B/ \mf m_B^i$ for $1\le i < e$.
        Therefore we have an    isomorphism 
        $
         B/\mathfrak{m}_A B \cong l[X]/(X^e)
        $
      of $l$-algebras, where $X$ corresponds to a uniformizer of $B$.

 The formal smoothness of $f\rightarrow B/\mathfrak{m}_A B$ gives a section $ s:l[X]/(X^{e}) \rightarrow l[X]/(X^{e+1}) $ of $ l[X]/(X^{e+1}) \rightarrow l[X]/(X^e) $.
 Then $s(\overline{X})^e=0$.
     If $e >1$, this yields a contradiction since $s(\overline{X})$ is a generator of the maximal ideal $(X)/(X^{e+1})$.
     Therefore we get $e=1$, and hence $k\rightarrow B/\mathfrak{m}_AB=l$ is formally smooth (resp., formally \'{e}tale).

(iv) $\Rightarrow$ (i): 
  $A\rightarrow B$ is flat since  $B$ is  torsion-free over the DVR $A$. 
Then the repeated use of  Lemma~\ref{lemma smooth lifting} shows that $A/\mathfrak{m}_A^n \rightarrow B/\mathfrak{m}_B^n$ is formally smooth (resp., formally \'{e}tale) for all $n\ge 1$.

Let $N\subset C$ satisfy $N^2=0$, and let $u_B:B\longrightarrow C/N $ be an $A$-algebra homomorphism such that $u_B(\mathfrak m_B^{n_0})=0$. 
If $u_A:A\to C$ is the structure map, then $u_A(\mathfrak m_A^{n_0})\subset N, $ and hence $u_A(\mathfrak  m_A^{2n_0})=0$. 
Thus $u_A$ and $u_B$ induce the following commutative diagram.

 \begin{equation*}
    \begin{tikzcd}
        B \arrow[r]   & B/\mathfrak{m}_B^{2n_0} \arrow[r, "\overline{u}_B"]  \arrow[rd, dashed, "t'"]& C/N    \\
        A \arrow[r] \arrow[u]  & A/\mathfrak{m}_A^{2n_0} \arrow[r, "\overline{u}_A"] \arrow[u] & C  \arrow[from=1-1, to=2-3, dotted, "t" ] \arrow[u] 
    \end{tikzcd}
\end{equation*}
 The finite-level lifting property gives $t'$, and its composite with $B\to B/\mathfrak m_B^{2n_0}$ is the required lift $t$.
In the formally étale case, every lift annihilates  $\mathfrak m_B^{2n_0}$ and hence factors through the same quotient, so  its uniqueness follows from that of $t'$.
\end{proof}

\subsection{Lifting homomorphisms and Cohen subrings}

\begin{lemma}\label{lemma lifting p-basis}
 Let $(R_0, \mathfrak{m}_0, k_0 )\subset (R, \mathfrak{m}_R, k )$ be  DVRs with $\mathfrak{m}_0R=\mathfrak{m}_R$. 
 Let $(A, \mathfrak{m}_A, l)$ be a local $R_0$-algebra which is complete for  the $\mathfrak{m}_A$-adic topology.
Suppose $k/k_0$ is separable,  and let  
$
\phi_{m}:R\longrightarrow A/\mathfrak{m}_A^{m}
$
be  a local $R_0$-algebra homomorphism for some  $m\ge 1$.
Fix a differential basis  $\underline{\beta}$   of $k/k_0$, representatives   $\underline{b} $ of  $\underline{\beta}$ in $R$,  and representatives $\underline{b}'$  of $\phi_m(\underline{b})$ in    $A$.
Then there is a unique $R_0$-algebra lifting 
$$
\Phi:R\longrightarrow A 
$$
 of $\phi_{m}$ satisfying   $\Phi(\underline{b})=\underline{b}'$.
   \end{lemma}

\begin{proof}

By Fact~\ref{lemma separable smooth criterion} (ii),  $\underline{\beta}$ is algebraically independent over $k_0$, which implies
that $\underline{b}$ is algebraically independent over $R_0$.
We consider the polynomial ring  $R_0[\underline{b}]$, which is a subring of $R$.
Let $\pi_0$ be a uniformizer of $R_0$. 
Then  $\pi_0R_0[\underline{b}]=(\pi_0)$  is a  prime ideal   and the localization $R_0[\underline{b}]_{(\pi_0)}$    is   a subring of $R$.

We show   $R_0[\underline{b}]_{(\pi_0)}$ is a DVR with  uniformizer $\pi_0$.
    $R_0$-freeness of $R_0[\underline{b}]$ shows   $\bigcap_n(\pi_0)^n=0$.
Hence, for   non-zero $f(\underline{b}) \in R_0[\underline{b}]_{(\pi_0)}$, there is unique $d$ such that $f(\underline{b}) \in (\pi_0)^d  \setminus (\pi_0)^{d+1} $.
Then $(f(\underline{b}) )  = (\pi_0)^d $, so  $R_0[\underline{b}]_{(\pi_0)}$ is a DVR.
The residue field of $R_0[\underline{b}]_{(\pi_0)}$ is $k_0(\underline{\beta})$.

By  the universal property of localization, the natural map  $R_0[\underline{b}]\rightarrow  R_0[\underline{b}]_{(\pi_0)}$ is formally \'{e}tale.
Since $k/k_0(\underline{\beta})$ is formally \'{e}tale by Fact~\ref{lemma separable smooth criterion},  the natural embedding $ R_0[\underline{b}]_{(\pi_0)}\rightarrow R $ is formally \'{e}tale for the $\mathfrak{m}_R$-adic topology by Lemma~\ref{lemma DVR smooth etale lift}. 
Therefore, by  Fact~\ref{lemma base change smooth},   the composite map 
$$R_0[\underline{b}]\rightarrow R_0[\underline{b}]_{(\pi_0)} \rightarrow R$$
is formally \'{e}tale for the $\mathfrak{m}_R$-adic topology.

It remains  to construct a map $\Phi:R\rightarrow A$ satisfying the desired properties.
There is a unique $R_0$-algebra homomorphism $  R_0[\underline{b}]\rightarrow A$ sending $\underline{b}$ to   $\underline{b}'$, which is compatible with  $\phi_m$.

Next,  we find inductively a unique map  $\phi_{n+1}:R\rightarrow A/\mathfrak{m}_A^{n+1}  $    such that the following diagram

 \begin{equation*}
    \begin{tikzcd}
        R   \arrow[rr,"\phi_n"]  \arrow[rrd,dotted,"\phi_{n+1}"]   & & A/\mathfrak{m}_A^n       \\
        R_0[\underline{b}] \arrow[u]  \arrow[r]  &   A  \arrow[r]    &  A/\mathfrak{m}_A^{n+1} \arrow[u]
    \end{tikzcd}
\end{equation*}
is commutative for $n\ge m$.

We get the desired map by defining
$$\Phi:=\varprojlim_n \phi_n:R\longrightarrow  A=\varprojlim_nA/\mathfrak{m}_A^{n}.$$ 
The uniqueness follows directly from the construction.
 \end{proof}

 \begin{remark}
  
In the lemma above, we must fix the representatives before claiming uniqueness.
 Also, a $p$-basis need not be a separating transcendence basis.
For example, the perfect closure of $\F_p(t)$ has empty $p$-basis over $\F_p$ and is not algebraic over $\F_p$.
The revised proof and Remark A of \cite[Lemma 4.2]{ADJ24} explain this distinction (see also \cite[Remark 2.28]{Dit26}).
Lemma~\ref{lemma lifting p-basis} uses formal \'etaleness,
so it does not require a separating transcendence basis.

\end{remark}

\begin{fact} [{\cite[Tag 03C3]{Stacks}}] \label{fact: existence DVR extension}
For a DVR  $(R_0, \mathfrak{m}_0, k_0)$ and an  extension $k/k_0$, one can always find a DVR $(R, \mathfrak{m}_R, k)$ extending  $(R_0, \mathfrak{m}_0, k_0)$ such that  $\mathfrak{m}_0R=\mathfrak{m}_R$.

\end{fact}

\begin{corollary}\label{corollary: Cohen subring existence}

 Let $(K,v)$ be a complete $\mathbb Z$-valued field of mixed characteristic $(0,p)$ with $v(p)=e$, and put $k=Kv$.
\begin{enumerate}
  \item
 There is a Cohen subring $C\subset \mathcal O_v$ inducing an isomorphism
$
C/pC\cong k.
$
  \item  
More generally, let $(R,\mf m_R,k_0)\subset \mathcal O_v$ be a complete   DVR   such that $\mf m_R=pR=R\cap\mathfrak m_v$.
If $k/k_0$ is separable, then $C$ in\/ {\rm (i)} may be chosen to contain $R$.
\end{enumerate}
For every such $C$ and   uniformizer $\pi$ of $\mathcal O_v$,
$$
\mathcal O_v=\bigoplus_{i=0}^{e-1}C\pi^i.
$$
In particular, $K/\operatorname{Frac}(C)$ is totally ramified of
degree $e$.
\end{corollary}

\begin{proof}
After completing the DVR supplied by Fact~\ref{fact: existence DVR extension}, the existence assertions follow from Lemma~\ref{lemma lifting p-basis}; see also the proof of \cite[Lemma~4.11]{ADJ24}.
For a uniformizer $\pi$ of $\mathcal O_v$, successive reduction modulo
$
p\mathcal O_v=\mathfrak m_v^e
$
and completeness give
$
\mathcal O_v=\sum_{i=0}^{e-1}C\pi^i.
$
The sum is direct since its nonzero summands have distinct valuations
modulo $e$.  
\end{proof}
 
 Without a separability assumption on the residue field extension, such a subring cannot always be chosen to contain a prescribed absolutely unramified sub-DVR, as in the following example.

\begin{example}
  
Set
$$
R=\widehat{\mathbb Z[T]_{(p)}}, \ \ A=R[X]/\bigl((X^p-T)^2-p\bigr),
$$
where $\widehat{(\bullet)}$ denotes $p$-adic  completion.
Then $R$ is a DVR with uniformizer $p$.
Write $x$ for the image of $X$ and put $\pi=x^p-T$, so  $\pi^2=p$.
Since $A$ is  a finitely generated free   $R$-module and $X^p-\overline T$ is irreducible over $\mb F_p(\overline T)$, the ring $A$ is a  $1$-dimensional regular local ring with maximal ideal $(\pi)$, i.e.,  a DVR. 
Its residue field is $\mathbb F_p(\overline x)$, where $\bar x^p=\overline T$, a purely inseparable extension of $R/pR=\mb F_p(\overline T)$ of degree $p$.

Suppose that a Cohen subring $C\subset A$ contains $R$ and induces an isomorphism $C/pC\to A/\pi A$.
Choose $b\in C$ with $b\equiv x\pmod{\pi A}$.
Since $pA=\pi^2A$, we have
$$
b^p-T\equiv x^p-T=\pi\not\equiv0\pmod{pA}.
$$
On the other hand, $\bar b^p=\overline T$ implies $b^p-T\in pC\subset pA$, a contradiction.
Therefore, there does not exist a Cohen subring  of $A$ which contains $R$.
\end{example}

\begin{remark} 
In  Lemma~\ref{lemma lifting p-basis}, suppose that $R_0$ and $R$ are complete and  of characteristic $0$,  and that   $k_0$  is a finite field.
 We apply the lemma with   $A=W(k)_{R_0}$  and $l=k$,  where $W(k)_{R_0}$ denotes  the ring of  $R_0$-ramified Witt vectors with coefficients in  $k$.
 We refer to \cite[Chapter 1]{Sch17}   for the theory of ramified Witt vectors.
 
 Since $W(k)_{R_0} $ is a complete local ring, it follows from  Lemma~\ref{lemma lifting p-basis} that there  exists a unique $R_0$-algebra homomorphism $R\rightarrow W(k)_{R_0}$ that sends $\underline{b}$ to $\underline{b}'$ and  induces the identity map on the residue fields. 
 This is a generalization of the classical fact that, given a field $k$ of characteristic $p$ and  its ring of Witt vectors $W(k)$, there exists a unique Cohen ring in $W(k)$ that contains a given lifting of a $p$-basis of $k$.
 
\end{remark}

The hypothesis $\mathfrak m_0R=\mathfrak m_R$ in Lemma~2.8 is
essential.
If $  \mathfrak m_0R=\mathfrak m_R^e $ with $e>1$, exact lifting may fail (see  \cite[Example~3.5]{LL21}). 
The root estimates (Lemmas~\ref{lem:modified Krasner} and \ref{lemma:different and krasner bound}) in the next subsection  will be used in Section~\ref{sec:embedding lemma} to obtain approximate liftings with explicit  bounds on the error of lifting.
 
\subsection{Totally ramified extensions and Krasner's lemma}

We recall some facts on a totally ramified extension and an Eisenstein polynomial.

\begin{fact}[{\cite[Chapter 2]{Lang94}}]\label{fact:totally ramified and Eisenstein}
 	A finite extension of henselian discretely valued fields   $(F,v_0)\subseteq (K,v)$ of mixed characteristic $(0,p)$ is totally ramified of degree $d$ if and only if $K$ is generated by a root $\pi \in \mathcal O_v$ of a monic Eisenstein polynomial $f(X)$ of degree $d$ over $\mc O_{v_0}$, i.e.,   
    $$f(X)=X^d+a_{d-1}X^{d-1}+\cdots+a_1X+a_0,\ \  a_i\in \mathfrak{m}_{v_0},  \ a_0\notin   \mathfrak{m}_{v_0}^2.$$
  In that case, 
  \begin{enumerate}
  \item 
  we have an $\mc O_{v_0}$-decomposition 
    $$
    \mc O_v =\bigoplus_{0 \le i < d} \mc O_{v_0} \pi^i;
    $$
    \item 
          if $p$ does not divide $d$, $\pi$ can be chosen so that $f(X)=X^d-\pi_0 u $  for a base uniformizer $\pi_0$ and  $u\in \mathcal O_{v_0}^{\times}$.
   
\end{enumerate}

\end{fact}

 \begin{fact}[Krasner's Lemma]\label{fact:Krasner}
Let $(K,v)$ be    henselian   and let $(K^{sep},v)$ be a separable closure of $K$. 
Let $\alpha,\beta\in K^{sep}$ be such that $$v(\beta-\alpha)>v(\alpha'-\alpha)$$ for any conjugate $\alpha'(\neq \alpha)$ of $\alpha$ over $K$. 
Then, $\alpha\in K(\beta)$.
\end{fact}

Let $(K, v)$ be a henselian field and  let $f(x)\in K[x]$ be   irreducible and separable of degree $d>1$.
For its roots $\alpha_1, \dots, \alpha_d$ in $K^{sep}$, set
$$
\Delta_v(f):=v(f'(\alpha_1))  \ \hbox { and } \  \kappa_v(f):= \mathrm{max}_{j\neq 1} \{ v(\alpha_1 -\alpha_j) \}.
$$
 Note that the definitions of $\Delta_v(f) $ and $\kappa_v(f)$ do not depend on the choice of $\alpha_1$.
The following lemmas will play an important role in Section~\ref{sec:embedding lemma}.
\begin{lemma}\label{lem:modified Krasner}

 Let $(F,v_0)\subseteq (K,v)$ be an extension of henselian  fields.
Let $f(X)\in F[X]$ be a monic irreducible separable polynomial and  $\alpha_1, \dots, \alpha_d$  roots of $f$ in $K^{sep}$.
  Choose $\beta\in K$, which is not necessarily algebraic or separable over $F$.
\begin{enumerate}
    \item 
 If  $ v(\beta-\alpha_i)>  \kappa_v(f) $, then $\alpha_i \in K$.  
 
    \item  
   
  If $v(f(\beta)) > \Delta_v(f) + \kappa_v(f)$, then  there is a unique root  $\alpha_i\in K$   such that  
  $$
  v(\beta-\alpha_i) = v (f(\beta)) - \Delta_v(f) > \kappa_v(f).
  $$
   \end{enumerate}
   \end{lemma}

\begin{proof}
(i) Let $g(X)\in K[X]$ be the minimal polynomial of $\alpha_i$ over $K$, which divides $f(X)$ in $K[X]$.
So, any  root $\alpha'(\neq \alpha_i)$ of $g$ is also a root  of $f$ and 
$$v(\beta-\alpha_i)>v(\alpha'-\alpha_i).$$ 
Thus, by Fact \ref{fact:Krasner}, we have  $\alpha_i \in K(\beta)=K$.  

(ii)
Without loss of generality, we may assume  $ v(\beta-\alpha_1) = \mathrm{max}_j \{  v(\beta-\alpha_j) \}$.
Then 
$$
v(\alpha_1 - \alpha_j)  \ge \mathrm{min} \{  v(\beta-\alpha_1),  v(\beta-\alpha_j) \}=  v(\beta-\alpha_j)
$$
for $j\neq 1$.
If $ v(\beta-\alpha_1) \le \kappa_v(f)$, 
\begin{align*}
v(f(\beta))  &=  \sum _{1\le j\le d}  v(\beta-\alpha_j)   \le  v(\beta-\alpha_1) +\sum_{j\neq 1} v(\alpha_1 -\alpha_j)   \le \kappa_v(f) +\Delta_v(f) ,
\end{align*}
which gives a contradiction.
Hence 
\begin{equation*}  
v(\beta-\alpha_1) > \kappa_v(f)  \  \hbox{ and }  \ v(\beta-\alpha_j ) = v(\alpha_1 -\alpha_j).
\end{equation*}
By (i), we have $\alpha_1 \in K $.
Furthermore, we have 
$$
v(f(\beta)) = v(\beta-\alpha_1) + \sum_{j\neq 1}v(\beta-\alpha_j) = v(\beta-\alpha_1) + \Delta_v(f).
$$
The uniqueness follows immediately from the ultrametric inequality.
\end{proof}
 
\begin{lemma} \label{lemma:different and krasner bound}
 
Let $f=\sum_{i=0}^d a_iX^i$ be an Eisenstein
polynomial of degree $d>1$ over a henselian discretely
valued field $F$ of mixed characteristic $(0,p)$.
Let $\alpha$ be a root of $f$ and put $K=F(\alpha)$.
Normalize the valuation $v$ on $K$ by $v(\alpha)=1$,
extend it to $K^{sep}$, and put $e_K=v(p)$.
Then we have
$$
\Delta_v(f)
=\min_{1\le i\le d}\{v(a_i)+e_Kv_p(i)+i-1\}
$$
and 
$$
\begin{aligned}
\Delta_v(f)&=d-1,
&\kappa_v(f)&=1
&&\text{if }p\nmid d,\\[4pt]
\Delta_v(f)&\le d-1+e_Kv_p(d),
&\kappa_v(f)&\le1+\frac{e_K}{p-1}
&&\text{if }p\mid d.
\end{aligned}
$$

\end{lemma}

\begin{proof} 

For the calculation of $\Delta_v(f)$ and the case $p\nmid d$,
see \cite[Chapter 3, Section 2]{Neu99} and \cite[Lemma~3.9]{LL21}.
The bound for $\kappa_v(f)$ when $p\mid d$ was proved by
Hou and Keating \cite[Lemma~2.2]{HK04} when $Kv$ is finite (cf. \cite[Ch. IV, \S 2, Exercise 3(c)]{Ser79} for the Galois case).
Their proof uses ramification theory.
Here we give an alternative, elementary proof that applies to arbitrary residue fields.

First we note that the identity in the lemma for $\Delta_v(f)$ shows 
\begin{equation}\label{equation different bound}
\Delta_v(f) \le v(ia_i \alpha^{i-1}), \quad 1\le i \le d.
\end{equation}

Using Taylor expansion, we have
$$ 
f(\alpha+Y)=\sum_{j=1}^d b_j Y^j, \ \  b_j= \sum_{i=j}^{d} \binom{i}{j}a_i \alpha^{i-j},
$$ 
where $\binom{\bullet}{\bullet}$ denotes binomial coefficient.
Using $\binom ij=(i/j)\binom{i-1}{j-1}$,
we obtain
\begin{align*}
v\left(\binom{i}{j}a_i \alpha^{i-j} \right) & =  v(i a_i \alpha^{i-1}) +v\left( \binom{i-1}{j-1} \right) -v(j)-(j-1) \\
& \ge \Delta_v(f) -e_Kv_p(j) -(j-1) ,
\end{align*}
where the last inequality follows from \eqref{equation different bound} and the fact that $\binom{i-1}{j-1}$ is an integer.
Thus 
\begin{equation} \label{equation coefficient bound}
v(b_j)\ge   \Delta_v(f)-(j-1)-e_Kv_p(j).
\end{equation}
Note that $v(b_1)=\Delta_v(f)$.

For $j\ge2$,  using  $p^{v_p(j)}-1\ge(p-1)v_p(j)$, we have  
\begin{equation}\label{equation p-valuation of integer}
(p-1)v_p(j)\le j-1  \hbox{, or equivalently, }\  \frac{j-1}{p-1} -v_p(j)\ge 0.
\end{equation}

We are ready to prove $\kappa_v(f) \le 1+ e_K/(p-1)$.
Put $y:=\alpha'-\alpha $ for some root $\alpha'(\neq\alpha)$ of $f$.
Then we have
$
0= y^{-1}f(\alpha+y) = b_1 + \sum_{j=2}^d b_j y^{j-1}.
$ 
If $v(y)>1+e_K/(p-1)$,   combining \eqref{equation coefficient bound} and \eqref{equation p-valuation of integer}, we obtain 
\begin{align*}
v(b_jy^{j-1} ) &=  v(b_j)+(j-1)v(y) \\
  & \ge \Delta_v(f) +(j-1)(v(y)-1) -e_K v_p(j)  \\
  & > \Delta_v(f) + e_K \left(   \frac{j-1}{p-1} -v_p(j)  \right) \\
  & \ge \Delta_v(f)
\end{align*}
for $j\ge2$.
Since $v(b_1)=\Delta_v(f)$,  we conclude that  $ v(0)=v\left( y^{-1}f(\alpha+y)\right) = \Delta_v(f) <\infty  $, which gives a contradiction.  
 \end{proof}

\subsection{Coarsening of valuation}\label{subsec:coarsening}
We review several facts on the coarsening of valuation. 
Let $(K,v)$ be a finitely ramified henselian valued field, let $\dot v$ be the finest equicharacteristic $0$ coarsening of $v$, and let $v^{\circ}$ be the valuation on the residue field $K\dot v$. 
So, for the smallest nontrivial convex subgroup $(v K)^{\circ}$ of $v K$, 
$$\dot vK= v K/(v K)^{\circ},\ v^{\circ} (K\dot v)= (v K)^{\circ}.$$ 
Let $\res_n^{\dot v}:\mathcal O_{\dot v}\rightarrow \mathcal O_{\dot v,n}$ and $\res_n^{v^{\circ}}:\mathcal O_{v^{\circ}}\rightarrow \mathcal O_{v^{\circ},n}$ be the residue maps. 

\begin{fact}[cf. {\cite[Remark 2.17]{LL21}}]\label{fact:coarsening and higher residue ring}
There is a canonical isomorphism $\theta_n:\mathcal O_{v,n}\rightarrow \mathcal O_{v^{\circ},n},x\mapsto x^{\circ}$ such that $\res^{v^{\circ}}_n\circ(\res^{\dot v}_1\restriction_{\mathcal O_v})=\theta_n\circ \res_n:\mathcal O_v\rightarrow \mathcal O_{v^{\circ},n}$.
\end{fact}

\begin{remark}\label{rem:coarsening and section}
Let $(K,v)$ be a finitely ramified henselian valued field, which is $\aleph_1$-saturated and let $\dot v$ be the finest coarsening of $v$ so that the residue field $(K\dot v, v^{\circ})$ is complete and of characteristic $0$. Then, there is an $\mathcal L_{ring}$-embedding $K\dot v\hookrightarrow \mathcal O_{\dot v}$ which is a section of the residue map $\res_{\dot v}$ (cf. \cite[Lemma~2.3]{AF16}). Note that this section map sends $(K\dot v,v^{\circ})$ to a valued subfield of $(K,v)$ because the valuation for a finitely ramified henselian valued field is $\mathcal L_{ring}$-definable (see Fact \ref{fact:definability of valuation ring}). As a valued subfield via this section map, $Kv=K\dot vv^{\circ}$ and $v^{\circ}K\dot v=(vK)^{\circ}$ so that $Kv/K\dot vv^{\circ}$ is separable and $vK/v^{\circ}K\dot v$ is torsion-free. Furthermore, this section map induces an inverse of the canonical isomorphism $\theta_n$ in Fact \ref{fact:coarsening and higher residue ring} for each $n\ge 1$.
\end{remark}

\subsection{Definability and interpretability on the residue ring sort}\label{subsec:definability on the residue ring}
We review several issues of definability and interpretability on the residue ring. And we introduce the notion of special formulas (Definition \ref{def:special formula}) to describe $\mathcal L_{val}$-formulas in a tuple of variables on the residue ring sort (cf. Theorem \ref{thm:formulas on the higher residue ring} and Proposition \ref{prop:equivlance between formula and special formula}).

\begin{fact}[cf. {\cite[Remark 2.2]{ADJ24}}]\label{fact:definability of valuation ring}
For a finitely ramified henselian valued field, the valuation ring and its maximal ideal are definable by parameter-free existential $\mathcal L_{ring}$-formulas. 
Indeed, given an initial ramification index $e$, take a positive integer $n>e$ coprime to $p$. 
Then, the valuation ring is defined by the formula $$\phi_{p,e}(x)\equiv \exists y(y^n=1+px^n),$$ and the maximal ideal is defined by the formula $$\psi_{p,e}(x)\equiv \exists y(y^n=1+\frac{1}{p}x^n).$$
\end{fact}

\begin{remark}\label{rem:interpretability of residue field}
Let $(K,v)$ be a finitely ramified valued field. 
Then, $\mathcal O_{v,n}$ interprets $Kv$ in the ring language $\mathcal L_{ring}$.
\end{remark}
\begin{proof}
Note that the maximal ideal of $\mathcal O_{v,n}$ is defined by the formula $x^n=0$. 
So, for an equivalence relation $E(x,y)\equiv (x-y)^n=0$, the underlying set of $Kv$ is given by $\mathcal O_{v,n}/E$. 
Also, it is not hard to see that given $a_1,a_2,b_1,b_2\in \mathcal O_{v,n}$ with $E(a_1,a_2)$ and $E(b_1,b_2)$, $E(a_1+b_1,a_2+b_2)$ and $E(a_1b_1,a_2b_2)$ hold, and so the $\mathcal L_{ring}$-structure on $Kv$ is interpreted in the $\mathcal L_{ring}$-structure $\mathcal O_{v,n}$.
\end{proof}

\begin{remark}\label{rem:separability of residue field via higer residue ring}
Let $(K,v)\subseteq (L,w)$ be valued fields of mixed characteristic $(0,p)$ and let $n\ge 1$.
Suppose $\mathcal O_{v,n}\preceq_{\exists}\mathcal O_{w,n}$ in $\mathcal L_{ring}$. 
Then, $Lw/Kv$ is separable.
\end{remark}
\begin{proof}
To show that $Lw/Kv$ is separable, that is, $Kv$ and $(Lw)^{(p)}$ are linearly disjoint over $(Kv)^{(p)}$, take $a_1,\ldots,a_k\in \mathcal O_v^{\times}$ such that $\res_1(a_1),\ldots,\res_1(a_k)$ are linearly dependent over $(Lw)^{(p)}$ and we will show that $\res_1(a_1),\ldots,\res_1(a_k)$ are linearly dependent over $(Kv)^{(p)}$.

Since $\res_1(a_1),\ldots,\res_1(a_k)$ are linearly dependent over $(Lw)^{(p)}$, there are $b_1,\ldots,b_k\in \mathcal O_w$ such that 
$$Lw\models \res_1(a_1)\res_1(b_1)^p+\cdots+\res_1(a_k)\res_1(b_k)^p=0.$$ 
Without loss of generality, we may assume that all of $b_i$'s are in $\mathcal O_w^{\times}$ and so $\res_1(b_i)\neq 0$. 
Recall that the maximal ideals of $\mathcal O_{v,n}$ and $\mathcal O_{w,n}$ are defined by the formula $x^n=0$. 
So we have
\begin{align*}
&Lw\models \sum_{i\le k}\res_1(a_i)\res_1(b_i)^p=0\\
 \Longleftrightarrow \ &\mathcal O_{w,n}\models (\sum_{i\le k}\res_n(a_i)\res_n(b_i)^p)^n=0\\
 \Longrightarrow \ &\mathcal O_{w,n}\models \exists \beta_1,\ldots,\beta_k\left(\bigwedge_{i\le k}\beta_i^n\neq 0\wedge (\sum_{i\le k}\res_n(a_i)\beta_i^p)^n=0\right).
\end{align*}
Since $\mathcal O_{v,n}\preceq_{\exists}\mathcal O_{w,n}$, there are $\beta_1',\ldots,\beta_k'\in \mathcal O_{v,n}$ such that
$$\mathcal O_{v,n}\models \bigwedge_{i\le k}(\beta_i')^n\neq 0\wedge (\sum_{i\le k}\res_n(a_i)(\beta_i')^p)^n=0,$$
which implies that there are $b_1',\ldots,b_k'\in \mathcal O_v^{\times}$ such that
$$\mathcal O_{v,n}\models (\sum_{i\le k}\res_n(a_i)\res_n(b_i')^p)^n=0.$$
So we have that $$Kv\models \sum_{i\le k}\res_1(a_i)\res_1(b_i')^p=0,$$
and $\res_1(a_1),\ldots,\res_1(a_k)$ are linearly dependent over $(Kv)^{(p)}$.
\end{proof}


Next, we introduce the notion of special formula in the residue ring sort $\mathcal O_n$. 
\begin{definition}\label{def:special formula}
Let $n\ge k\ge 1$ be integers. 
A {\em special $\mathcal O_k$-formula at level $n$} is a $\mathcal L_{val}$-formula of the form: $$\res^n_k(\theta')(x)\equiv\exists x'\left(\theta'(x')\wedge \res^n_k(x')=x\right),$$ where $x'$ and $x$ are tuples  of variables of the same length on the sorts $\mc O_n$ and $\mc O_k$, respectively, and  $\theta'(x')$ is a $\mathcal L_{ring}$-formula.
It is called an  \emph{$\exists^+$-special $\mathcal O_k$-formula at   level $n$} if $\theta'$ is a positive existential $\mathcal L_{ring}$-formula.
\end{definition}

\begin{remark}\label{rem:kernel_speical formula}
Let $n\ge k\ge 1$.
\begin{enumerate}
    \item The relation $\res^n_k(x)=\res^n_k(y)$ on the sort $\mc O_n$ is defined by a positive $\mathcal L_{ring}$-formula $$\exists s\exists t(s^n=0\wedge x-y=s^k t).$$
    \item For a $\mathcal L_{ring}$-formula $\theta'(x')$ with a tuple $x'$ of variables on the sort $\mc O_n$, let $$\hat \theta'(x')\equiv \exists y'\left(\theta'(y')\wedge \res^n_k(x')=\res^n_k(y')\right),$$ which is also a positive formula if $\theta'$ is a positive formula. Then two special $\mc O_k$-formulas $\res^n_k(\theta')(x)$ and $\res^n_k(\hat \theta')(x)$ are equivalent. 
\end{enumerate}
\end{remark}

\begin{remark}\label{rem:special formula under boolean combination}
Let $n\ge k\ge1$ be integers. 
Any finite Boolean combination of special $\mathcal O_k$-formulas at level $n$ is equivalent to a special $\mathcal O_k$-formula at level $n$. 
Moreover, any finite conjunction or disjunction of $\exists^+$-special $\mathcal O_k$-formulas at level $n$ is equivalent to an $\exists^+$-special $\mathcal O_k$-formula at level $n$.
\end{remark}

\begin{proof}
Take $\mathcal L_{ {ring}}$-formulas $\theta_0'(x')$ and $\theta_1'(x')$, where $x'$ is a tuple of variables on the sort $\mathcal O_n$. By Remark \ref{rem:kernel_speical formula}(ii), we may assume that for $i=0,1$,  $$\theta'_i(x')\models\forall y'\left(\res^n_k(y')=\res^n_k(x')\rightarrow \theta'_i(y')\right).$$
Then, formulas $\res^n_k(\theta'_0)(x)\vee \res^n_k(\theta'_1)(x)$, $\res^n_k(\theta'_0)(x)\wedge \res^n_k(\theta'_1)(x)$, and $\neg\res^n_k(\theta'_0)(x)$ are equivalent to special $\mathcal O_k$-formulas at level $n$, $\res^n_k(\theta'_0\vee \theta'_1)(x)$, $\res^n_k(\theta'_0\wedge \theta'_1)(x)$, and $\res^n_k(\neg\theta'_0)(x)$ respectively, which comes from the fact that for a surjective map $f:X\rightarrow Y$ and for any subsets $U, V$ of $Y$,
\begin{itemize}
    \item $f^{-1}(U\cap V)=f^{-1}(U)\cap f^{-1}(V)$,
    \item $f^{-1}(U\cup V)=f^{-1}(U)\cup f^{-1}(V)$, 
    \item $f^{-1}(Y\setminus U)=X\setminus f^{-1}(U)$, and
    \item $f[f^{-1}(U)]=U$.
\end{itemize}
\end{proof}

\section{Approximate lifting and lifting numbers}\label{sec:embedding lemma}

\subsection{Approximate lifting}

\begin{notation}
Let $(F,v_0)\subseteq(K,v)$ and $(L,w)$ be complete discretely valued fields.
Let $\varphi:\mathcal O_{v,m}\rightarrow \mathcal O_{w,m}$ be a homomorphism and let $\Phi:K\rightarrow L$ be an isometric embedding.
For $m  \ge n\ge 1$, we write $\res^m_n(\varphi):\mathcal O_{v,n}\rightarrow \mathcal O_{w,n}$ and $\res_n(\Phi):\mathcal O_{v,n}\rightarrow \mathcal O_{w,n}$ for the natural maps induced by $\varphi$ and $\Phi$, respectively. 

  Let $\Psi:F\rightarrow L$ be an isometric embedding. 
We say that $\Psi$ is {\em compatible with $\varphi$} if  
$$\varphi(\res_m(x))=\res_m(\Psi(x)), \quad  x\in \mathcal O_{v_0}. $$
For $f(X)=\sum_i a_iX^i\in F[X]$,  put $\Psi(f)(X):=\sum_i \Psi(a_i)X^i\in L[X] $.

\end{notation}

 We first prove a relative lifting result with a fixed coefficient embedding.
 The following proposition strengthens   \cite[Theorem~3.7]{LL21}.
  
\begin{proposition}  \label{prop: general lifting}
Let $(F,v_0)\subseteq(K,v)$ and $(L,w)$ be complete discretely valued fields of mixed characteristic $(0,p)$. 
Suppose that $(F,v_0)\subseteq(K,v)$ is totally ramified of degree $d>1$, and that $v$ and $w$ are normalized at their uniformizers.
Let $n\geq1$, let $\varphi:\mathcal{O}_{v,n}\to\mathcal{O}_{w,n}$ be a homomorphism, and let $\Psi:F\to L$ be an isometric embedding compatible with $\varphi$.
 Let $f(X)\in F[X]$ be an Eisenstein polynomial  such that $K=F(\pi)$ for a root $\pi$ of $f$.

 \begin{enumerate}
     \item 
 Suppose $ n>\Delta_w(\Psi(f))+\kappa_w(\Psi(f)) $.
 Then there is a unique isometric embedding $\Phi:K\to L$ extending $\Psi$
and satisfying
$$
\mr{res}_k(\Phi) =  \mr{res}^n_k(\varphi), \ \ 1\leq k\leq n-\Delta_w(\Psi(f)).
$$
For any $\beta\in\mathcal O_w$ representing $\varphi(\operatorname{res}_n(\pi))$, the element $\Phi(\pi)$ is the unique
root $\alpha$ of $\Psi(f)$ satisfying
$$
w(\alpha-\beta)>\kappa_w(\Psi(f)).
$$

     \item 
     Suppose
$$
n>
\begin{cases}
\dfrac{e_L}{e_K}d, & \hbox{when }\  p\nmid d,\\ e_L\left(\dfrac{d}{e_K}+v_p(d)+\dfrac{1}{p-1}\right),& \hbox{when }\  p\mid d.
\end{cases}
$$
 Then the conclusion of\/ {\rm (i)} holds. 
 In particular,
$$
\mr{res}_k(\Phi)=\operatorname{res}^{n}_{k}(\varphi),   \ \ 1\leq k\leq n-e_L\left(\frac{d-1}{e_K}+v_p(d)\right).
$$    
 \end{enumerate}
 
\end{proposition}

\begin{proof}
 
Suppose $ n>\Delta_w(\Psi(f))+\kappa_w(\Psi(f))$.
Let $\beta\in\mathcal O_w$ be a representative of $\varphi(\operatorname{res}_n(\pi))$.
By the compatibility of $\Psi$ with $\varphi$, we have $ \operatorname{res}_n(\Psi(f)(\beta)) = \varphi(\operatorname{res}_n(f(\pi)))=0 $, and hence
$$
\Psi(f)(\beta)\in\mathfrak m_w^n.
$$
 By Lemma~\ref{lem:modified Krasner}, there is   a unique root $\pi'$ of $\Psi(f)$ in $ L$ such that
\begin{equation} \label{equation beta value bound}
w(\beta-\pi')
=w(\Psi(f)(\beta))-\Delta_w(\Psi(f)) \ge n-\Delta_w(\Psi(f)) > \kappa_w(\Psi(f)).
\end{equation} 
Define $\Phi:K\rightarrow L$ by extending $\Psi$ with $\pi\mapsto\pi'$.
Noting $w(p)= (e_L/e_K) v(p)$, we have $w(\Psi(a))=(e_L/e_K) v(a)$ for $a\in F$ and  $w(\pi')=e_L/e_K$.
We conclude that $\Phi$ is an isometric embedding with 
$$w(\Phi(x))=(e_L/e_K) v(x)$$
for all  $x\in K$.

Since $\beta-\pi'$ divides  $\beta^i-(\pi')^i$, it follows from \eqref{equation beta value bound} and the compatibility of  $\Psi$ with $ \varphi $   that 
\begin{equation} \label{equation restriction bound}
\mr{res}_k(\Phi)=\mr{res}^{n}_{k}(\varphi),  \ \  1\leq k\leq n-\Delta_w(\Psi(f)).
\end{equation}
Since $n-\Delta_w(\Psi(f)) > \kappa_w(\Psi(f))$, for  $\beta\in\mathcal O_w$ representing $\varphi(\operatorname{res}_n(\pi))$, the element $\Phi(\pi)=\pi'$ is the unique root $\alpha$ of $\Psi(f)$ satisfying
$ w(\alpha-\beta)>\kappa_w(\Psi(f)) $ by Lemma~\ref{lem:modified Krasner}.

To prove the uniqueness of $\Phi$, let $\Phi_1:K\to L$ be an isometric embedding extending $\Psi$ and satisfying \eqref{equation restriction bound}.
If $\beta$ represents $\varphi(\operatorname{res}_n(\pi))$, we have
$$
w(\Phi_1(\pi)-\beta)\ge   n-\Delta_w(\Psi(f)) >\kappa_w(\Psi(f)).
$$
Moreover, $\Phi_1(\pi)$ is a root of $\Psi(f)$. 
Hence the uniqueness assertion in Lemma~\ref{lem:modified Krasner} gives
$$
\Phi_1(\pi)=\pi'.
$$
It follows that $\Phi_1=\Phi$,  which finishes the proof of  (i).

The uniform bounds  in (ii) follow directly from  
$$
\Delta_w(\Psi(f)) \le  e_L\left(\frac{d-1}{e_K}+v_p(d)\right) \ \hbox{ and }  \ \kappa_w(\Psi(f))\le
\begin{cases}
e_L/e_K,&p\nmid d,\\
e_L/e_K+e_L/(p-1),&p\mid d,
\end{cases}
$$
by  Lemma~\ref{lemma:different and krasner bound}.
\end{proof}

In the following  theorem, we prove  the existence of approximating lifting.

\begin{theorem} \label{thm: approximate lifting}
Let $(K,v)$ and $(L,w)$ be complete $\Z$-valued fields of mixed
characteristic $(0,p)$ with $e_K=e_L=e$. Let
$\varphi:\mathcal O_{v,n}\to\mathcal O_{w,n}$ be a ring homomorphism.
Assume
$$
n \ge
\begin{cases}
1, & \hbox{ when }\  e=1,\\ 
e+1 ,& \hbox{ when }\  e>1 \hbox{  and } p\nmid e, \\
e(1+v_p(e)) +\left\lfloor\dfrac{e}{p-1}\right\rfloor+1  , & \hbox{ when }\  e>1 \hbox{ and } p\mid e,
\end{cases}
$$
where $\left\lfloor \cdot \right\rfloor$ denotes the greatest integer function.
 Then there is an isometric embedding $\Phi:K\to L$ such that
\begin{equation*} 
\res_k(\Phi)=\res_k^n(\varphi)
\ \
\text{ for }\ 1\le k\le n-(e-1+e v_p(e)).
\end{equation*}
If $\varphi$ is an isomorphism, $\Phi$ is an isomorphism.
 
\end{theorem}

\begin{proof}

Choose a Cohen subring $C\subset\mathcal O_v$ as in Corollary~\ref{corollary: Cohen subring existence}, and put  $ F:=\operatorname{Frac}(C) $.
Applying Lemma~\ref{lemma lifting p-basis} with
$
R_0=\mathbb Z_p,\  R=C,\  A=\mathcal O_w,\  m=n
$
to the composite map
$$
C\longrightarrow\mathcal O_v\xrightarrow{\res_n}
\mathcal O_{v,n}\xrightarrow{\varphi}\mathcal O_{w,n},
$$
we obtain an embedding
$
\Psi_0:C\longrightarrow\mathcal O_w
$
compatible with $\varphi$.
Let
$
\Psi:F\longrightarrow L
$
be its extension.

If $e=1$, then $\mathcal O_v=C$ by Corollary~\ref{corollary: Cohen subring existence}, and we take $\Phi=\Psi$. 
Assume that $e>1$.
Let $\pi$ be a uniformizer of $K$,
and let $f\in C[X]$ be its Eisenstein polynomial over $F$. 
Applying Proposition~\ref{prop: general lifting}~(ii) with $ d=e$ and $e_K=e_L=e $, we obtain an isometric embedding
$$
\Phi:K\longrightarrow L
$$
extending $\Psi$ such that $ \mr{res}_k(\Phi) = \mr{res}^n_k(\varphi)  \ \text{ for }\ \ 1\leq k\leq n-(e-1+ev_p(e)) $.

Suppose finally that $\varphi$ is an isomorphism, and put $ C'=\Psi_0(C) $.
Then $C'/pC'\cong Lw$, so Corollary~\ref{corollary: Cohen subring existence} gives $ [L:\mr{Frac}(C')]=e $.
On the other hand, $ [\Phi(K):\operatorname{Frac}(C')] = [K:F] = e $.
Hence  $\Phi$ is an isomorphism.
\end{proof}
\noindent
 In Theorem~\ref{thm: approximate lifting},  $\Phi$ is not necessarily unique (see Remark~\ref{rem:functorial-lifting}).
 By Proposition~\ref{prop: general lifting}, $\Phi$ is unique if its restriction to the Cohen subring $C$ is fixed.

\subsection{Lifting numbers}

\begin{definition}\label{def:lifting-number}
Let $p$ be a prime and let $e\geq 1$. 
The \emph{lifting number} $N_p(e)$ is the smallest positive integer $n$ such that
$$
\mathcal O_{v,n}\cong \mathcal O_{w,n}
\quad\Longrightarrow\quad
(K,v)\cong (L,w)
$$
for all complete $\mathbb Z$-valued fields $(K,v)$ and $(L,w)$ of mixed characteristic $(0,p)$ and initial ramification index $e$.
 \end{definition}
 \noindent
  Note that  we may use complete DVRs in the definition since a  DVR determines its fraction field.

\begin{remark}\label{rem:functorial-lifting}
In \cite[Definition~4.1 and Theorem~4.4]{LL21}, where the residue fields
are assumed to be perfect, the lifting number is defined categorically and
shown to be finite by constructing a lifting functor. 

Such a functorial formulation does not extend to arbitrary residue fields.
For example, let
$$
A=\widehat{\mb Z_p[T]_{(p)}}, \ \ R=A[\pi]=A[X]/(X^e-p).
$$
The residue field of $A$ is $\mathbb F_p(T)$. 
For   $n\geq 1$, the map, defined by
$$
T\longmapsto T+p^n, \ \ \pi\longmapsto \pi,
$$
is an  automorphism $\sigma_n$ of $R$, which satisfies
$
(\sigma_n-\operatorname{id}_R)(R) \subset p^nR \subset \mathfrak m_R^n.
$
Thus the   functor $\overline{\res_n}$, induced by $\res_n$, is not faithful.
However,  the existence of a lifting functor whose composite with $\overline{\res_n}$ is naturally isomorphic to the identity would imply faithfulness. 
Therefore, throughout the paper, the lifting number means the weaker invariant   in Definition~\ref{def:lifting-number}.
\end{remark}

We note that  Theorem~\ref{thm: approximate lifting} says that the  lifting number is always finite.
Throughout the remainder of this subsection, we prove the bounds in Theorem~\ref{thm: approximate lifting} are exactly lifting numbers.
\begin{proposition}\label{prop:tame-pair}
Let $e>1$ and $p\nmid e$. 
There are complete $\Z$-valued fields $(K_1,v_1)$ and $(K_2,v_2)$ of initial ramification index $e$ with finite residue fields such that their residue rings of length $e$ are isomorphic, but $X^e+p$ has a root in $K_1$ and none in $K_2$.
In particular their valuation rings are not isomorphic.
\end{proposition}

\begin{proof}
Choose $h\ge 1$ with $e\mid p^h-1$.
Let $F/\Q_p$ be the unramified extension of degree $h$, with residue field $k$.
Choose $u\in\mathcal O_F^\times$ such that $\bar{u}\notin (k^\times)^e$. 
Put  $K_1=F(\sqrt[e]{-p})$ and  $K_2=F(\sqrt[e]{-pu})$.
Then $K_1$ and $K_2$ satisfy the desired assertion by the theory of tamely ramified extensions.
Cf. \cite[Proposition~4.5 and Lemma~4.6]{LL21}.
\end{proof}

For the wild case $p \mid e$,  we start with the following elementary lemma.

\begin{lemma}\label{lem: wild case lemma}
  \begin{enumerate}
      \item 
 Let $(K, v) $ be a discretely valued field with  initial ramification index $e=v(p)$.
    Suppose $p-1\nmid e$.
    For $x\in K$ with  $t=v(x-1)>0$,  we have
 $$
v(x^p-1)=
\begin{cases}
pt, & \hbox{ if }\  t\le e/(p-1),\\ 
e+t ,& \hbox{ if }\  t\ge e/(p-1). \\
\end{cases}
 $$
      \item 
     For $r>0$, the splitting field of  $f(X):=X^{p^r} +p\in \Q_p[X]$  is totally ramified over $\Q_p$.
     Furthermore, if $e=p^rm$ with $p\nmid m$ and $(p-1) \mid e$, then $\Q_p(\sqrt[e]{-p}, \mu_{p^r})/\Q_p$ is totally ramified. 
  \end{enumerate}
    
\end{lemma}

 \begin{proof}
    (i) The case $x=1$ is obvious.
    Assume $x\neq1$.
    The binomial expansion  shows that  the first and last terms of 
     $$
 x^p-1 =    (1+(x-1))^p-1=  p(x-1)+\sum _{1< i< p} \binom{p}{i}(x-1)^i+(x-1)^p
     $$ 
     have values   $e+t$ and  $pt$, respectively.
Every intermediate term has greater value than $\min\{pt,e+t\}$.
Since $p-1 \nmid e$, we have $e+t\neq pt$, so $v(x^p-1)=\min\{pt,e+t\}$.
The assertion follows by noting  that  for a nonnegative real number $c$,
$\, c\le e/(p-1) \ \Leftrightarrow \  pc \le e+c$.

(ii) 
Put $\alpha:= (-p)^{1/p^r}$.
 Let $F$ be the splitting field of $f(X)$ over $\Q_p$ and let $F'/\Q_p$ be the maximal abelian subextension of $F/\Q_p$.
 Then $F'=\Q_p(\mu_{p^r}, \alpha^{p^{r-i}})$ for some $0\le i \le r$ by Kummer theory.

\begin{itemize}
  \item 
Suppose $p>2$.
If $i\ge 1$, we have that $\Q_p(\alpha^{p^{r-1}})/\Q_p $ is abelian, and hence $\Q_p(\alpha^{p^{r-1}})$ contains $\mu_p$.
Since $[\Q_p(\alpha^{p^{r-1}}): \Q_p]=p$, this is a contradiction and so $i=0$ and $F'=\Q_p(\mu_{p^r})$.

  \item 
Suppose $p=2$.
If $i\ge 2$, then $\Q_2(\alpha^{2^{r-2}})/\Q_2$ is abelian of degree $4$, and hence $\zeta_4=\sqrt{-1}$ is contained in $\Q_2(\alpha^{2^{r-2}})$.
The nontrivial automorphism $\sigma$ of $ \Gal(\Q_2(\alpha^{2^{r-2}})/\Q_2(\mu_4))$ fixes $\mu_4$, hence  $\sigma(\alpha^{2^{r-2}}) = -\alpha^{2^{r-2}}$.
Therefore we obtain $\alpha^{2^{r-1}}=\sqrt{-2} \in \Q_2(\mu_4)$, which is a contradiction.
Hence we get $i \le 1$ and $F'\subset \Q_2(\mu_{2^r}, \sqrt{-2}) \subset \Q_2(\mu_{2^c})$ where $c:= \mr{max}\{r, 3 \} $ and we note $\Q_2(\zeta_8)= \Q_2(\zeta_4, \sqrt{-2})$.
\end{itemize}
In either case, $F'/\Q_p$ is totally ramified.
Since the maximal unramified subextension of $F/\Q_p$ is abelian, we conclude that $F/\Q_p$ is totally ramified.

Now put $\pi:=(-p)^{1/e}$.
We may assume  $\alpha=\pi^m$. 
Since $\mathbb Q_p(\mu_p)=\mathbb Q_p(\pi^{e/(p-1)})$, the fields $F$ and $\mathbb Q_p(\pi)$ are totally ramified extensions of $\mathbb Q_p(\alpha,\mu_p)$ of degrees dividing $p^{r-1}$ and equal to $m/(p-1)$, respectively. 
Since $p\nmid m/(p-1)$, it follows that  $\mathbb Q_p(\pi,\mu_{p^r}) /\Q_p$ is totally ramified.
 \end{proof}

 Write
$
e=p^{r'}m,\   \  p\nmid m.
$
 Also we simply write
\begin{equation*} 
j'=\left\lfloor\frac{e}{p -1}\right\rfloor, \ \  n_0=e(1+r')+j'.
\end{equation*}
We construct nonisomorphic DVRs with isomorphic   residue rings at length $n_0$.
We treat the cases  $(p-1)\nmid e$ and $(p-1)\mid e$ separately.

\begin{proposition}\label{prop: wild nondivisible case}
Suppose $p\mid e$. 
Let $\pi_1$ and $\pi_2$ be roots of
$$
f(X)=X^e+p\ \hbox{ and } \ g(X)=X^e+p+p^{r'+1}X^{j'},
$$
respectively.
Put 
$$R_i:=\Z_p[\pi_i], \ \ K_i:=\Q_p(\pi_i), \ \ \mf m_i=(\pi_i) , \quad i=1, 2.$$ 
Then   $R_1/\mathfrak m_1^{n_0}\cong R_2/\mathfrak m_2^{n_0}$, but $R_1\not\cong R_2$.
Moreover, $f(X)$ has no root in $K_2$.
\end{proposition}

\begin{proof} 
Since $f$ and $g$ are Eisenstein, $R_1$ and $R_2$ are DVRs with the same residue field $\F_p$ and the same initial ramification index $e$. 
Note that $g$ is not monic when $p=2$.
Since
$
f(\pi_2)=-p^{r'+1}\pi_2^{j'}\in\mathfrak m_2^{n_0},
$
the substitution $\pi_1\mapsto\pi_2$ induces    $R_1/\mathfrak m_1^{n_0} \cong R_2/\mathfrak m_2^{n_0}$.
Let $v$ be a   valuation on $\Q_p^{sep}$ with  $v(\pi_1)=1$ and $v(p)=e$.
We prove $g(X)$ has no root in $K_1$.
Then $R_1 \not\cong R_2$ and $f$ has no root in $K_2$.

\medskip
\textbf{Case 1}: suppose $p-1\nmid e$.
 Suppose $g(X)$ has a root $\alpha$ in $K_1$. 
Put  $y:=\alpha/\pi_1$ and $z:=y^m$.
Then $y$ is  a unit and 
\begin{equation}\label{eq: wild nondivisible 1}
y^e-1 = p^{r'}\pi_1^{j'} y^{j'}.
\end{equation}
Also, reduction in the residue field $\mb F_p$ gives $\bar z^{p^{r'}}=1$, so  $\bar z=1$. 
Moreover, we have $z\ne1$ and
\begin{equation}\label{eq: wild nondivisible 2}
v(z^{p^{r'}}-1)=er'+j'
\end{equation}
by \eqref{eq: wild nondivisible 1}.
 Put $t=v(z-1)$.
 
 \begin{itemize}
     \item 

If $t \le j' <e/(p-1)$, then $ v(z^p-1) =pt \le pj'$ and   $ v(z^{p^{a+1}}-1) \le v(z^{p^{a}}-1) +e\,$ for $\,1\le a \le r'\,$ by Lemma~\ref{lem: wild case lemma}.
Thus
$$
v(z^{p^{r'}}-1)\le pj'+e(r'-1)<er'+j',
$$
where the last inequality follows from  $(p-1)j'<e$. 
     \item 
     If $t \ge j'+1  >e/(p-1)$, then $ v(z^{p^a}-1)= t+ea$ for $1\le a \le r'$ by Lemma~\ref{lem: wild case lemma}.
     Thus  
$$
v(z^{p^{r'}}-1)=   er' +  t>  er' +j'.
$$ 
 \end{itemize}
Both possibilities contradict \eqref{eq: wild nondivisible 2}.
Thus $g(X)$ has no root in $K_1$.

\medskip
\textbf{Case 2}:
suppose $p-1 \mid e$.
 Put $E:=K_1(\mu_{p^{r'}})$, whose residue field is $\F_p$ by Lemma~\ref{lem: wild case lemma}.
Suppose that $\beta\in K_1$ is a root of $g$.
Since $g$ is Eisenstein, we have $v(\beta)=1$ and
\begin{equation}\label{eq: wild divisible 1}
f(\beta)=-p^{r'+1}\beta^{j'}, \  \ \frac{\beta^e}{p} = -1-p^{r'}\beta^{j'} \equiv -1 \pmod{\mathfrak m_E}.
\end{equation}
In particular, $v(f(\beta))=n_0$.

Choose a root $\alpha$ of $f$ such that $v(\beta-\alpha) = \mr{max} \{ v(\beta-\alpha'): f(\alpha')=0 \}$.
For every other root $\alpha'$ of $f$, the ultrametric inequality gives $ v(\alpha-\alpha') \ge   v(\beta-\alpha') $.
Since $f'(\alpha)=e\alpha^{e-1}$ and $v(\alpha)=1$, it follows that
$$
n_0=v(f(\beta)) \leq v(\beta-\alpha)+v(f'(\alpha)) = v(\beta-\alpha)+e(r'+1)-1.
$$
Thus $v(\beta-\alpha)\geq j'+1$.

We claim that $\alpha\in E$.
If we write $\alpha= \pi_1 \xi \chi$ with $\xi \in \mu_m$ and $\chi \in \mu_{p^{r'}}$, then we have $\bar{\xi}=\overline{\beta\pi^{-1}_1}\in \F_p^\times$ since $v(\beta-\alpha)\geq j'+1>1$ and $\bar{\chi}=1$.
Thus we have $\xi\in \mu_{p-1}\subset \Q_p$, and hence $\alpha\in E$ and   $ s:=(\alpha-\beta)/\beta^{j'+1} \in \mathcal O_E $.

Consider the following polynomial
$$
H(X):=\frac{f(\beta+\beta^{j'+1}X)}{p^{r'+1}\beta^{j'}} = -1+\sum_{i=1}^{e}b_iX^i, \ \ b_i:=\binom ei\frac{\beta^{e+j'(i-1)}}{p^{r'+1}}.
$$
Then $H(s)=0$.
Using
$\binom ei=(e/i)\binom{e-1}{i-1}$ and $e=j' (p-1)$, we obtain
$$
v(b_i) = e\left(v_p\binom ei-r'\right)+j'(i-1) \ge  j' \bigl(i-1-(p-1)v_p(i)\bigr) \ge 0.
$$
The last expression is positive except for $i=1, p \ $ since  equality in $i-1\ge p^{v_p(i)}-1\ge(p-1)v_p(i)$ holds only at $i=1,p$.
Moreover, we have
$$
b_1=m\frac{\beta^e}{p}, \ \ b_p=\frac{\binom ep}{p^{r'-1}}  \left(\frac{\beta^e}{p}\right)^2, \ \ \frac{\binom ep}{p^{r'-1}} =  m\binom{e-1}{p-1}\equiv m\pmod p.
$$
Since $\beta^e/p\equiv -1\pmod{\mathfrak m_E}$ by \eqref{eq: wild divisible 1}, we have $\overline{b_1} = -\bar{m} $ and $\overline{b_p}=\bar m$. 
Therefore we get
$$
H(X)\in\mathcal O_E[X] \ \hbox{ and } \ \overline H(X)=\bar m(X^p - X)- 1 \in \F_p[X].
$$
Since $\overline{H}(x) =-1$  for all $x\in \F_p$, this  contradicts $\overline{H}(\bar{s})=0$.
 Thus $g(X)$ has no root in $K_1$. 
\end{proof}

\begin{theorem} \label{thm: lifting number}
For every prime $p$ and positive integer $e$, lifting numbers are given as follows:
$$
N_{p} (e)=
\begin{cases}
1,&  \hbox{ if }\ e=1,\\
e+1,&  \hbox{ if } \ e>1\ \hbox{ and }\  p\nmid e,\\
e(1+v_p(e))+\lfloor e/(p-1)\rfloor+1,&  \hbox{ if } \  p\mid e.
\end{cases}
$$
 
\end{theorem}

\begin{proof}
This follows from    Propositions~\ref{prop:tame-pair} and  \ref{prop: wild nondivisible case}, together with  Theorem~\ref{thm: approximate lifting}.
\end{proof}

\section{Ax-Kochen-Ershov principle via residue rings}\label{sect:AKE via higher residue rings}
In this section, we prove several Ax-Kochen-Ershov principles for relative completeness, relative existential completeness, and existential closedness using the pure ring structures on the residue rings. First, we provide an embedding lemma via residue rings, analogous to \cite[Lemma 4.12]{ADJ24}.

\begin{lemma}\label{lem:embedding lemma over substr}
Let $(K,v)$ and $(L,w)$ be two extensions of a valued field $(K_0,v_0)$, all of which are henselian and finitely ramified of mixed characteristic $(0,p)$ having the same initial ramification index $e$. Suppose that
\begin{itemize}
	\item $v K/v_0 K_0$ is torsion-free;
	\item $Kv/K_0v_0$ is separable; and
	\item $(K_0,v_0)$ and $(K,v)$ are $\aleph_1$-saturated, and $(L,w)$ is $|K|^+$-saturated.
\end{itemize}
Let $n$ be a positive integer. Then, if there are an $\mathcal L_{ring}$-embedding $\varphi_{\mathcal O_n}:\mathcal O_{v,n}\rightarrow \mathcal O_{w,n}$ over $\mathcal O_{v_0,n}$ and an $\mathcal L_{oag}$-embedding $\varphi_{\Gamma}:vK\rightarrow wL$ over $v_0K_0$, then there is an $\mathcal L_{val}$-embedding $K\rightarrow L$ over $K_0$ inducing $\varphi_{\mathcal O_n}$ and $\varphi_{\Gamma}$.
\end{lemma}
\begin{proof}
We follow the construction in the proof of \cite[Lemma~4.12]{ADJ24}. For completeness, we give the details of its proof using Lemma \ref{lemma lifting p-basis} instead of \cite[Lemma~4.11]{ADJ24}.  Let $\dot v_0$, $\dot v$, and $\dot w$ be the finest  equicharacteristic $0$ coarsenings of $v_0$, $v$, and $w$, respectively. 
  By the saturation assumption,   $ (K_0\dot v_0,v_0^\circ),\  (K\dot v,v^\circ),  $  and $ (L\dot w,w^\circ) $ are complete  discretely valued fields with the same initial ramification index. 
Thus, under the natural inclusions, we have
$$
\Gamma^\circ
: =(v_0K_0)^\circ
=(vK)^\circ
=(wL)^\circ ,\ \ \mathfrak m_{v_0^\circ}\mathcal O_{v^\circ} =\mathfrak m_{v^\circ}, \ \ \mathfrak m_{v_0^\circ}\mathcal O_{w^\circ} =\mathfrak m_{w^\circ}.
$$
Let $\varphi^{\circ}_{\mathcal O_n}:\mathcal O_{v^{\circ},n}\rightarrow \mathcal O_{w^{\circ},n}$ be an $\mathcal L_{ring}$-embedding over $\mathcal O_{v^{\circ}_0,n}$ induced from $\varphi_{\mathcal O_n}$ by Fact~\ref{fact:coarsening and higher residue ring}.
 By Lemma~\ref{lemma lifting p-basis} with a local $\mathcal O_{v_0^{\circ}}$-algebra homomorphism $\varphi_{\mc O_n}^{\circ}\circ\res^{v^{\circ}}_n:\mathcal O_{v^{\circ}}\rightarrow \mathcal O_{w^{\circ},n}$, there is an embedding $\Phi^{\circ}:(K\dot v,v^{\circ})\rightarrow (L\dot w, w^{\circ})$ over $(K_0\dot v_0,v_0^{\circ})$ inducing $\varphi^{\circ}_{\mathcal O_n}$.

By the saturation assumptions and the torsion-freeness of $vK/v_0K_0$, as in \cite[Lemmas~7.9 and 7.10]{vdD14},  there is  a cross-section $ s_0:v_0K_0\longrightarrow K_0^\times$, a cross-section $ s:vK\longrightarrow K^\times $ extending $s_0$, and a partial cross-section $ t:\varphi_\Gamma(vK)\longrightarrow L^\times $ whose restriction to $v_0K_0$ is $s_0$.
The exact sequence
$$
1\longrightarrow\mathcal O_{v^\circ}^{\times} \longrightarrow K^\times/(1+\mathfrak m_{\dot v}) \longrightarrow vK
\longrightarrow 0
$$
is split by $s$, where the last map sends the class of $a$ to $v(a)$. Here we identify $\mathcal O_{v^\circ}^{\times}$ with its natural image in $K^\times/(1+\mathfrak m_{\dot v})$. We obtain an embedding
$$
\psi:K^\times/(1+\mathfrak m_{\dot v}) \longrightarrow L^\times/(1+\mathfrak m_{\dot w}), \  \  u\cdot \overline{s(\gamma)}  \longmapsto \Phi^\circ(u)\cdot \overline{t(\varphi_\Gamma(\gamma))} \ \ \bigl(u\in\mathcal O_{v^\circ}^{\times},\ \gamma\in vK\bigr).
$$
  By construction, $\psi$ is over $ K_0^\times/(1+\mathfrak m_{\dot v_0}) $ and induces the prescribed map $\varphi_\Gamma$ on the full value groups.

Let $\epsilon$ be the   least positive value in $\Gamma^\circ$ and put $\pi=s_0(\epsilon)$.
The group $(K\dot v)^\times$ is generated by $\mathcal O_{v^\circ}^{\times}$ and the residue class $\pi^\circ$ of $\pi$. 
Since $\Phi^\circ$ fixes $\pi^\circ$ and the cross-sections agree on $v_0K_0$, the restriction of $\psi$ to
$$
(K\dot v)^\times
\simeq\mathcal O_{\dot v}^{\times}/(1+\mathfrak m_{\dot v})
$$
agrees with $\Phi^\circ$. 
Let $ \varphi_\Gamma^\circ:\dot vK\longrightarrow\dot wL $  be the embedding induced by $\varphi_\Gamma$. 
We have the following commutative diagram with exact rows
$$
\begin{tikzcd}
1 \arrow[r] & (K\dot v)^{\times} \arrow[r] \arrow[d, "\Phi^{\circ}"]& K^{\times}/(1+\mathfrak{m}_{\dot v}) \arrow[r] \arrow[d, "\psi"] & \dot vK  \arrow[r] \arrow[d, "\varphi_{\Gamma}^{\circ}"] &0\\
1 \arrow[r] & (L\dot w)^{\times} \arrow[r]& L^{\times}/(1+\mathfrak{m}_{\dot w}) \arrow[r] & \dot wL  \arrow[r] & 0,
\end{tikzcd}
$$
compatible with the corresponding sequence for $(K_0,\dot v_0)$.
The left-hand map comes from the field embedding $\Phi^\circ$, the middle map is a group embedding, and the right-hand map is an ordered-group embedding. 
Thus these maps give an embedding of the co-structures of $(K,\dot v)$ and $(L,\dot w)$ over that of $(K_0,\dot v_0)$, in the terminology of \cite[Section~4]{Bel99}.

By  \cite[Theorem~2.1]{Bas91}, in the formulation of \cite[Théorème~4.2]{Bel99}, $\psi$ is induced by an embedding
$$
\Phi:(K,\dot v)\longrightarrow(L,\dot w)
$$
over $(K_0,\dot v_0)$. 
The required $|K|$-pseudocompleteness of $(L,\dot w)$ for \cite[Theorem~2.1]{Bas91} follows from the $|K|^+$-saturation of $(L,w)$.

By the construction of $\psi$, we have
$$
w(\Phi(a))=\varphi_\Gamma(v(a))
\qquad(a\in K^\times).
$$
Thus $\Phi$ is an embedding of $(K,v)$ into $(L,w)$ inducing $\varphi_\Gamma$. It also induces $\Phi^\circ$ on the cores, hence $\varphi_{\mathcal O_n}^\circ$ on their residue rings of length $n$. By Fact \ref{fact:coarsening and higher residue ring}, the induced map on $\mathcal O_{v,n}$ is therefore $\varphi_{\mathcal O_n}$. This gives the required $\mathcal L_{\mathrm{val}}$-embedding over $K_0$.  
\end{proof}
\noindent The following embedding lemma is analogous to \cite[Lemma 4.13]{ADJ24}.
\begin{lemma}\label{lem:relative embedding absolute}
Let $(K,v)$ and $(L,w)$ be two finitely ramified henselian valued fields of mixed characteristic $(0,p)$ having the same initial ramification index $e\ge 1$. Suppose $(L,w)$ is $|K|^+$-saturated.

Let $k\ge 1 $  and  $n\ge \mr {max}\{k+e-1+ev_p(e), N_p(e) \}.$ Let $\varphi:\mathcal O_{v,n}\rightarrow \mathcal O_{w,n}$ be a homomorphism and let $\varphi_{\Gamma}:vK\rightarrow wL$ be an $\mathcal L_{oag}$-embedding sending the minimal positive element in $vK$ to the minimal positive element in $wL$. Then, the pair of embeddings $\res_k(\varphi):\mathcal O_{v,k}\rightarrow \mathcal O_{w,k}$ and $\varphi_{\Gamma}$ is induced by a $\mathcal L_{val}$-embedding $K\rightarrow L$.
\end{lemma}
\begin{proof}
We follow the construction in the proof of \cite[Lemma~4.13]{ADJ24}. For completeness, we give a detailed proof using Theorem~\ref{thm: approximate lifting} instead of \cite[Lemma~4.6]{ADJ24}.

Let $(K^*,v^*)$, $(L^*,w^*)$, $\varphi^*:\mathcal O_{v^*,n}\rightarrow \mathcal O_{w^*,n}$, and $\varphi_{\Gamma}^*:v^*K^*\rightarrow w^*L^*$ be elementary extensions of $(K,v)$, $(L,w)$, $\varphi$, and $\varphi_{\Gamma}$ such that $(K^*,v^*)$ and $(L^*,w^*)$ are $\aleph_1$-saturated. The assumption $ n\ge \mr {max}\{k+e-1+ev_p(e), N_p(e) \}$ implies that  $n\ge N_p(e)$ and $k\le n-(e-1+e v_p(e))$. Since $(K^*,v^*)$ and $(L^*,w^*)$ are $\aleph_1$-saturated, by Theorem~\ref{thm: approximate lifting}, the homomorphism $\res_k(\varphi^*):\mathcal O_{v^*,k}\rightarrow \mathcal O_{w^*,k}$ is induced by an $\mathcal L_{val}$-embedding $(K^*\dot v^*,(v^*)^{\circ})\hookrightarrow (L^*\dot w^*,(w^*)^{\circ})$.

By Remark \ref{rem:coarsening and section}, via section maps of $\res_{\dot v^*}$ and $\res_{\dot w^*}$, we consider $K^*\dot v^*$ and $L^*\dot w^*$ as valued subfields of $(K^*,v^*)$ and $(L^*,w^*)$, which have the same residue rings as $(K^*,v^*)$ and $(L^*,w^*)$, respectively and we have that $K^*v^*=K^*\dot v^*(v^*)^{\circ}$ and $(v^*)^{\circ}K^*\dot v^*$ is the   smallest nontrivial  convex subgroup of $v^*K^*$.
Also, via the embedding $(K^*\dot v^*,(v^*)^{\circ})\hookrightarrow (L^*\dot w^*,(w^*)^{\circ})$, we consider $K^*\dot v^*$ as a valued subfield of $(L^*,w^*)$. 
Therefore, for $(K_0^*,v_0^*):=(K^*\dot v^*,(v^*)^{\circ})$, we have
\begin{itemize}
    \item $K^*v^*/K_0^*v_0^*$ is separable and $v^*K^*/v_0^*K_0^*$ is torsion-free; and
    \item $\varphi_{\Gamma}^*:v^*K^*\rightarrow w^*L^*$ is over $v_0^* K_0^*$,
\end{itemize}
and since  $\mathcal O_{v^*,k}=\mathcal O_{v_0^*,k}$, $\mathcal O_{w^*,k}=\mathcal O_{(w^{*})^{\circ},k}$, and the homomorphism $\res_k(\varphi^*):\mathcal O_{v^*,k}\rightarrow \mathcal O_{w^*,k}$ is induced by the embedding $(K^*\dot v^*,(v^*)^{\circ})\hookrightarrow (L^*\dot w^*,(w^*)^{\circ})$, we have
\begin{itemize}
    \item  $\res_k(\varphi^*):\mathcal O_{v^*,k}\rightarrow \mathcal O_{w^*,k}$ is over $\mathcal O_{v_0^*,k}$.
\end{itemize}
After replacing by elementary extensions, we may assume that $K^*$, $K_0^*$, and $L^*$ are saturated enough to apply Lemma \ref{lem:embedding lemma over substr} with the maps $\res_k(\varphi^*)$ and $\varphi_{\Gamma}^*$ and so there is an $\mathcal L_{val}$-embedding $K^*\hookrightarrow L^*$ inducing $\res_k(\varphi^*)$ and $\varphi_{\Gamma}^*$.

Then, the composite embedding $(K,v)\hookrightarrow (K^*,v^*)\hookrightarrow (L^*,w^*)$ induces the map $vK\hookrightarrow wL\hookrightarrow w^*L^*$ on the value groups and the map $\mathcal O_{v,k}\hookrightarrow \mathcal O_{w,k}\hookrightarrow \mathcal O_{w^*,k}$ on the residue rings, coming from   $\varphi_{\Gamma}$  and $\res_k(\varphi)$   respectively. 
This implies that $(K,v)$ can be embedded into the elementary extension $(L^*,w^*)$ of $(L,w)$ as $\mathcal L_{val}(vK,\mathcal O_{v,k})$-structures and so by the $|K|^+$-saturation of $(L,w)$, $(K,v)$ can be $\mathcal L_{val}(vK,\mathcal O_{v,k})$-embedded into $(L,w)$.
\end{proof}

\subsection{Relative completeness}\label{subsect:relative completeness}
By combining Remark \ref{rem:interpretability of residue field} and \cite[Theorem 5.1]{ADJ24}, we first prove the relative model completeness via the residue rings.
\begin{theorem}\label{thm:relative model completeness}
Let $(K,v)\subseteq (L,w)$ be two finitely ramified henselian valued fields of the same initial ramification index $e$. Let $k$ be a positive integer. 
Suppose $\mathcal O_{v,k}\preceq \mathcal O_{w,k}$ and $vK\preceq wL$ in $\mathcal L_{ring}$ and $\mathcal L_{oag}$ respectively. 
Then, $(K,v)\preceq (L,w)$ in $\mathcal L_{val}$.
\end{theorem}
\begin{proof}
By Remark \ref{rem:interpretability of residue field}, we have that $Kv\preceq Lw$ in $\mathcal L_{ring}$.
So, by \cite[Theorem 5.1]{ADJ24}, we have $(K,v)\preceq (L,w)$ in $\mathcal L_{ring}$ and since each sort $\mathcal O_n$ and maps $\res^m_n$, $\res_n$ are  interpretable in the language $\mathcal L_{ring}$ by Fact \ref{fact:definability of valuation ring}, $(K,v)\preceq (L,w)$ in $\mathcal L_{val}$.
\end{proof}

Using Lemma~\ref{lem:relative embedding absolute} and Theorem~\ref{thm:relative model completeness}, we provide a description of types of elements of the residue rings and value groups, analogous to \cite[Theorem 5.2]{ADJ24}.
\begin{theorem}\label{thm:relative completeness}
Let $(K,v)$ and $(L,w)$ be finitely ramified henselian valued fields of mixed characteristic $(0,p)$ having the same initial ramification index $e\ge 1$.
Let $k\ge 1 $  and 
$n\ge \mr {max}\{k+e-1+ev_p(e), N_p(e) \}.$
Let $\alpha$ and $\beta$ be finite tuples of the same length  in  $\mathcal O_{v,k}$ and $\mathcal O_{w,k}$ and let $\zeta$ and $\xi$ be finite tuples of the same length in $vK$ and $wL$, respectively.
Then, the following are equivalent:
\begin{enumerate}
	\item $\tp_{(K,v)}^{\mathcal L_{val}}(\alpha,\zeta)=\tp_{(L,w)}^{\mathcal L_{val}}(\beta,\xi)$.
	\item $\alpha$ and $\beta$ satisfy the same special $\mathcal O_k$-formulas  at level $n$, and $\tp_{vK}^{\mathcal L_{oag}}(\zeta)=\tp_{wL}^{\mathcal L_{oag}}(\xi)$.
	\item 
   There are elementary extensions $(K', v') \succeq (K, v)$ and $(L', w')\succeq (L, w)$, and tuples $\alpha'$  and $\beta'$ in $\mathcal O_{v',n}$ and $\mathcal O_{w',n}$, respectively, such that 
    $$\tp_{\mathcal O_{v',n}}^{\mathcal L_{ring}}(\alpha')=\tp_{\mathcal O_{w',n}}^{\mathcal L_{ring}}(\beta'),\  \ \res_k^n(\alpha')=\alpha,\ \ \res_k^n(\beta')=\beta,$$ 
	and $\tp_{vK}^{\mathcal L_{oag}}(\zeta)=\tp_{wL}^{\mathcal L_{oag}}(\xi)$.
     
\end{enumerate}
\end{theorem}
\begin{proof}
It is clear that (i) $\Rightarrow $ (ii). 
For the direction (ii) $\Rightarrow$ (iii), suppose the condition (ii) holds. 
By taking elementary extensions properly, we may assume that $(K,v)$ and $(L,w)$ are $\aleph_1$-saturated. Then, given any $\alpha'\in \mathcal O_{v,n}$ with $\res^n_k(\alpha')=\alpha$, by (ii), compactness, and $\aleph_1$-saturation, there is a tuple $\beta'\in \mathcal O_{w,n}$ such that $\tp_{\mathcal O_{v,n}}^{\mathcal L_{ring}}(\alpha')=\tp_{\mathcal O_{w,n}}^{\mathcal L_{ring}}(\beta'),\ \res_k^n(\alpha')=\alpha,\ \res_k^n(\beta')=\beta$.\\

Now, we will show the direction (iii) $\Rightarrow$  (i).
Without loss of generality, we may assume that $K=K'$ and $L=L'$. Suppose there are tuples $\alpha'$ and $\beta'$ in $\mathcal O_{v,n}$ and $\mathcal O_{w,n}$ respectively such that $$\tp_{\mathcal O_{v,n}}^{\mathcal L_{ring}}(\alpha')=\tp_{\mathcal O_{w,n}}^{\mathcal L_{ring}}(\beta'),\ \res_k^n(\alpha')=\alpha,\ \res_k^n(\beta')=\beta,$$ 
and $\tp_{vK}^{\mathcal L_{oag}}(\zeta)=\tp_{wL}^{\mathcal L_{oag}}(\xi)$. After replacing $(L,w)$ by a $|K|^+$-saturated elementary extension, we may suppose that we have an elementary $\mathcal L_{ring}$-embedding $\varphi:\mathcal O_{v,n}\rightarrow \mathcal O_{w,n}$ sending $\alpha'$ to $\beta'$ and an elementary $\mathcal L_{oag}$-embedding $\varphi_{\Gamma}:vK\rightarrow wL$ sending $\zeta$ to $\xi$. By Lemma~\ref{lem:relative embedding absolute}, there exists an $\mathcal L_{val}$-embedding $(K,v)\hookrightarrow (L,w)$ inducing $\res_k(\varphi)$ and $\varphi_{\Gamma}$ and by Theorem~\ref{thm:relative model completeness} this embedding is elementary. In particular, the embedding $\res_k(\varphi)$ sends $\alpha$ to $\beta$ and so we have $\tp_{(K,v)}^{\mathcal L_{val}}(\alpha,\zeta)=\tp_{(L,w)}^{\mathcal L_{val}}(\beta,\xi)$.
\end{proof}

\begin{corollary}\label{cor:relative complete}
Let $(K,v)$ and $(L,w)$ be finitely ramified henselian valued fields of mixed characteristic $(0,p)$ having the same initial ramification index $e$.
Put $n\ge N_p(e)$. 
Then, $$(K,v)\equiv (L,w)\  \ \Longleftrightarrow \ \ vK\equiv wL,\ \mathcal O_{v,n}\equiv \mathcal O_{w,n}.$$
Moreover, $N_p(e)$ is the smallest integer such that the equivalence  holds   for all models of $T_{p,e}$.
\end{corollary}
\begin{proof}
It is enough to show the right-to-left direction. Suppose $vK\equiv wL$ and $\mathcal O_{v,n}\equiv \mathcal O_{w,n}$. If we let $k=1$ in Theorem~\ref{thm:relative completeness}, we get $\mr {max}\{k+e-1+ev_p(e), N_p(e) \}=N_p(e)$. By Theorem \ref{thm:relative completeness} with $k:=1$ and empty tuples $\alpha$ and $\beta$, $\tp_{(K,v)}^{\mathcal L_{val}}(\alpha)=\tp_{(L,w)}^{\mathcal L_{val}}(\beta)$ and so $(K,v)\equiv (L,w)$. The minimality of $N_p(e)$ follows from the examples constructed in Propositions~\ref{prop:tame-pair} and \ref{prop: wild nondivisible case}.
\end{proof}

\begin{corollary}\label{cor:transfer of decidability of full theories}
Let $(K,v)$ be a finitely ramified henselian valued field of mixed characteristic $(0,p)$ having the initial ramification index $e$. 
Put $n=N_p(e)$. 
Then, $\Th^{\mathcal L_{val}}(K,v)$ is decidable if and only if $\Th^{\mathcal L_{oag}}(vK)$ and $\Th^{\mathcal L_{ring}}(\mathcal O_{v,n})$ are decidable.
\end{corollary}
\begin{proof}
Let $T_{p,e}$ be the $\mathcal L_{val}$-theory of all finitely ramified henselian valued fields of mixed characteristic $(0,p)$ having the initial ramification index $e$, which is computably axiomatizable. Let $T_0:=T_{p,e}\cup \Th^{\mathcal L_{val}}(K,v)$ and $T_1:=T_{p,e}\cup\Th^{\mathcal L_{oag}}(vK)\cup\Th^{\mathcal L_{ring}}(\mathcal O_{v,n})$. Then, by Corollary \ref{cor:relative complete}, $T_0\models T_1$ and $T_1\models T_0$. So, $\Th^{\mathcal L_{val}}(K,v)$ is decidable if and only if $\Th^{\mathcal L_{oag}}(vK)$ and $\Th^{\mathcal L_{ring}}(\mathcal O_{v,n})$ are decidable.
\end{proof}

\subsection{Relative existential completeness}\label{subsect:relative existential completness}
Given a tuple $\alpha$ of elements in $\mathcal O_{v,n}$, we write $\exists\mh\tp_{(K,v)}^{\mathcal L_{val}}(\alpha)$ for the existential type of $\alpha$ in $(K,v)$ and write $\exists^+\mh\tp_{\mathcal O_{v,n}}^{\mathcal L_{ring}}(\alpha)$ for the positive existential type.

\begin{theorem}\label{thm:relative existential completeness}
Let $(K,v)$ and $(L,w)$ be finitely ramified henselian valued fields of mixed characteristic $(0,p)$ having the same initial ramification index $e$. Let $k\ge 1$  and  $n\ge \mr {max}\{k+e-1+ev_p(e), N_p(e) \}$. Let $\alpha$ and $\beta$ be finite tuples of $\mathcal O_{v,k}$ and $\mathcal O_{w,k}$ having the same length respectively. 
Then, the following are equivalent:
\begin{enumerate}
	\item $\exists\mh\tp_{(K,v)}^{\mathcal L_{val}}(\alpha)\subseteq \exists\mh\tp_{(L,w)}^{\mathcal L_{val}}(\beta)$.
	\item Given an $\exists^+$-special $\mathcal O_k$-formula $\theta(x)$ at level $n$, 
    $$(K,v)\models \theta(\alpha)\ \Longrightarrow \  (L,w)\models \theta(\beta).$$
	\item 
      There are elementary extensions $(K', v') \succeq (K, v)$ and $(L', w')\succeq (L, w)$, and tuples $\alpha'$  and $\beta'$ in $\mathcal O_{v',n}$ and $\mathcal O_{w',n}$, respectively, such that 
    $$\exists^+\mh\tp_{\mathcal O_{v',n}}^{\mathcal L_{ring}}(\alpha') \subset  \exists^+\mh\tp_{\mathcal O_{w',n}}^{\mathcal L_{ring}}(\beta'),\  \ \res_k^n(\alpha')=\alpha,\ \ \res_k^n(\beta')=\beta.$$ 
\end{enumerate} 
\end{theorem}
\begin{proof}
Since $\exists^+$-special formulas are existentially $\mathcal L_{val}$-definable, it is not hard to see (i) $\Rightarrow$ (ii) $\Rightarrow$ (iii).

\medskip
It remains to prove the direction (iii) $\Rightarrow $ (i). 
Without loss of generality, we may assume $K=K'$ and $L=L'$. Suppose that there are tuples $\alpha'$ and $\beta'$ in $\mathcal O_{v,n}$ and $\mathcal O_{w,n}$ respectively such that $$\exists^+\mh\tp_{\mathcal O_{v,n}}^{\mathcal L_{ring}}(\alpha')\subseteq\exists^+\mh\tp_{\mathcal O_{w,n}}^{\mathcal L_{ring}}(\beta'),\ \res_k^n(\alpha')=\alpha,\ \res_k^n(\beta')=\beta.$$ 
We mimic the proof of \cite[Proposition 5.6]{ADJ24} using Theorem \ref{thm: approximate lifting} instead of \cite[Lemma 4.6]{ADJ24}.
After taking elementary extensions, we may assume that $(K,v)$ is $\aleph_1$-saturated and $(L,w)$ is $|K|^+$-saturated. Since $\exists^+\mh\tp_{\mathcal O_{v,n}}^{\mathcal L_{ring}}(\alpha')\subseteq\exists^+\mh\tp_{\mathcal O_{w,n}}^{\mathcal L_{ring}}(\beta')$, there is an $\mathcal L_{ring}$-homomorphism $\mathcal O_{v,n}\rightarrow \mathcal O_{w,n}$ sending $\alpha'$ to $\beta'$. 
Also, by saturation, $(K\dot v,v^{\circ})$ and $(L\dot w,w^{\circ})$ are complete $\mathbb Z$-valued fields of the same initial ramification index $e$.
By Fact \ref{fact:coarsening and higher residue ring}, the homomorphism from $\mathcal O_{v,n}\rightarrow \mathcal O_{w,n}$ induces a homomorphism $\mathcal O_{v^{\circ}, n}\rightarrow \mathcal O_{w^{\circ},n}$, and by Theorem \ref{thm: approximate lifting}, there is an embedding $K\dot v\rightarrow L\dot w$ inducing a homomorphism $\mathcal O_{v^{\circ},k}\rightarrow \mathcal O_{w^{\circ},k}$ sending $\alpha^{\circ}$ to $\beta^{\circ}$. 
So, we have that $\exists\mh\tp_{(K\dot v,v^{\circ})}^{\mathcal L_{val}}(\alpha^{\circ})
\subseteq \exists\mh\tp_{(L\dot w,w^{\circ})}^{\mathcal L_{val}}(\beta^{\circ})$.

By a standard argument (cf. \cite[Lemma~2.3]{AF16}), there are existentially closed embeddings $(K\dot v,v^{\circ})\rightarrow (K,v)$ with $\alpha^{\circ}\mapsto\alpha$ and $(L\dot w,w^{\circ})\rightarrow (L,w)$ with $\beta^{\circ}\mapsto \beta$ and so 
$$\exists\mh\tp_{(K\dot v,v^{\circ})}^{\mathcal L_{val}}(\alpha^{\circ})=\exists\mh\tp_{(K,v)}^{\mathcal L_{val}}(\alpha),\ \exists\mh\tp_{(L\dot w,w^{\circ})}^{\mathcal L_{val}}(\beta^{\circ})=\exists\mh\tp_{(L,w)}^{\mathcal L_{val}}(\beta).$$ 
Thus, we have
$$\exists\mh\tp_{(K,v)}^{\mathcal L_{val}}(\alpha)=\exists\mh\tp_{(K\dot v,v^{\circ})}^{\mathcal L_{val}}(\alpha^{\circ})\subseteq \exists\mh\tp_{(L\dot w,w^{\circ})}^{\mathcal L_{val}}(\beta^{\circ})=\exists\mh\tp_{(L,w)}^{\mathcal L_{val}}(\beta).$$
\end{proof}

\begin{corollary} \label{cor: relative existential}
Let $(K,v)$ and $(L,w)$ be finitely ramified henselian valued fields of mixed characteristic $(0,p)$ having the same initial ramification index $e$. Put $n= N_p(e)$. Then, 
$$\Th_{\exists}^{\mathcal L_{val}}(K,v)\subseteq \Th_{\exists}^{\mathcal L_{val}}(L,w) \ \Longleftrightarrow  \ \Th_{\exists^+}^{\mathcal L_{ring}}(\mathcal O_{v,n})\subseteq \Th_{\exists^+}^{\mathcal L_{ring}}(\mathcal O_{w,n}).$$
In particular, $(K,v) \equiv _{\exists} (L, w) \ \Longleftrightarrow \  \mc O_{v, n} \equiv _{\exists^+} \mc O_{w, n}$.
\end{corollary}
\begin{proof}
By choosing $k=1$ and empty tuples $\alpha$ and $\beta$, it comes from Theorem \ref{thm:relative existential completeness}.
\end{proof}

\begin{theorem}\label{thm:Hilbert's 10th via the higher valued field}
Let $(K,v)$ be  a finitely ramified henselian valued field of mixed characteristic $(0,p)$ having the initial ramification index $e$. Put $n= N_p(e)$. Then, $\Th_{\exists}^{\mathcal L_{val}}(K,v)$ is decidable if and only if $\Th_{\exists^+}^{\mathcal L_{ring}}(\mathcal O_{v,n})$ is decidable.
\end{theorem}
\begin{proof}
It is standard and similar to the proof of Corollary \ref{cor:transfer of decidability of full theories}. 
Let $T_{p,e}$ be the $\mathcal L_{val}$-theory of all finitely ramified henselian valued fields of mixed characteristic $(0,p)$ having the initial ramification index $e$, which is computably axiomatizable. 
Let $T_{\exists}^{(K,v)}$ be the set of all existential $\mathcal L_{val}$-sentences in $\Th^{\mathcal L_{val}}_{\exists}(K,v)$ and negations of all existential $\mathcal L_{val}$-sentences not in $\Th^{\mathcal L_{val}}_{\exists}(K,v)$.
Let $T_{\exists^+}^{\mathcal O_{v,n}}$ be the set of all positive existential $\mathcal L_{ring}$-sentences in $\Th^{\mathcal L_{ring}}_{\exists}(\mathcal O_{v,n})$ and negations of all positive existential $\mathcal L_{ring}$-sentences not in $\Th^{\mathcal L_{ring}}_{\exists}(\mathcal O_{v,n})$. 
Let $T_0:=T_{p,e}\cup T^{(K,v)}_{\exists}$ and $T_1:=T_{p,e}\cup T^{\mathcal O_{v,n}}_{\exists^+}$. 

We will show $T_0\models T_1$ and $T_1\models T_0$. 
Given a valued field $(L, w)$, if $(L,w)\models T_0$, then $(L,w)\equiv_{\exists}(K,v)$ and by Theorem \ref{thm:relative existential completeness}, $\mathcal O_{w,n}\equiv_{\exists^+}\mathcal O_{v,n}$. 
So, $(L,w)\models T_1$, which implies $T_0\models T_1$. 
Conversely, if $(L,w)\models T_1$, then $\mathcal O_{w,n}\equiv_{\exists^+}\mathcal O_{v,n}$ and by Theorem \ref{thm:relative existential completeness} again, $(L,w)\equiv_{\exists}(K,v)$. 
So, $(L,w)\models T_0$, which implies $T_1\models T_0$.

Since $T_0\models T_1$ and $T_1\models T_0$, we have that $\Th^{\mathcal L_{val}}_{\exists}(K,v)$ is decidable if and only if $\Th^{\mathcal L_{ring}}_{\exists^+}(\mathcal O_{v,n})$ is decidable.
\end{proof}

\subsection{Existential closedness}\label{subsect:existential closedness}
Using Lemma \ref{lem:embedding lemma over substr}, we have an AKE principle for relative existential closedness relative to residue rings of higher length.
\begin{theorem}\label{theorem:relative existential closedness}
Let $(K,v)\subseteq (L, w)$ be finitely ramified henselian valued fields of mixed characteristic $(0,p)$ having the same initial ramification index $e$. Let $n$ be a positive integer. If $\mathcal O_{v,n}\preceq_{\exists} \mathcal \mathcal O_{w,n}$ in $\mathcal L_{ring}$ and $vK\preceq_{\exists} wL$ in $\mathcal L_{oag}$, then $(K,v)\preceq_{\exists} (L,w)$ in $\mathcal L_{val}$.
\end{theorem} 
\begin{proof}
We mimic the proof of \cite[Theorem 5.14]{ADJ24}. 
Taking ultrapowers properly, we may assume that $(K,v)$ and $(L,w)$ are $\aleph_1$-saturated, and consider an $|L|^+$-saturated elementary extension $(K^*,v^*)$ of $(K,v)$.

Since $\mathcal O_{v,n}\preceq_{\exists} \mathcal O_{w,n}$ and $vK\preceq_{\exists} wL$, by Remark \ref{rem:separability of residue field via higer residue ring}, $Lw/Kv$ is separable.
 Moreover,  $w L/v K$ is torsion-free since  $vK\preceq_{\exists}wL$.
Also, since $\mathcal O_{v,n}\preceq_{\exists} \mathcal O_{w,n}$ and $(K^*,v^*)$ is $|L|^+$-saturated, there is an $\mathcal L_{ring}$-embedding $\mathcal O_{w,n}\rightarrow \mathcal O_{v^*,n}$ over $\mathcal O_{v,n}$. 
Similarly, since $vK\preceq_{\exists}wL$, there is an $\mathcal L_{oag}$-embedding $wL\rightarrow v^* K^*$ over $vK$. 
Therefore, by Lemma \ref{lem:embedding lemma over substr}, there is an embedding $(L,w)\rightarrow (K^*,v^*)$ over $(K,v)$, which implies that $(K,v)\preceq_{\exists} (L,w)$.
\end{proof}

\section{Induced structure on residue rings}\label{sec:induced}
In this section, we show that any $\mathcal L_{val}$-definable set on the sort $\mathcal O_k$ is exactly defined by a $\mathcal L_{ring}$-formula in a finitely ramified henselian valued field of mixed characteristic, which is analogous to \cite[Theorem 6.2]{ADJ24}.

\begin{theorem}\label{thm:formulas on the higher residue ring}
Let $(K,v)$ be a finitely ramified henselian valued field and let $n$ be a positive integer.
Let $x$ be a tuple of variables on the sort $\mathcal O_n$ and $y$ be a tuple of variables on the sort $\Gamma$.
Then, any $\mathcal L_{val}(K)$-formula $\theta(x,y)$ is equivalent in $K$ to a boolean combination of  $\mathcal L_{ring}(\mathcal O_{v,n})$-formulas in $x$ and   $\mathcal L_{oag}(vK)$-formulas in $y$. 
In particular, any $\mathcal L_{val}(K)$-definable subset of $\mathcal O_{v,n}$ is defined by a $\mathcal L_{ring}(\mathcal O_{v,n})$-formula.
\end{theorem}
\begin{proof}
We follow the proof strategy of \cite[Theorem~6.2]{ADJ24} using Theorem \ref{thm:relative model completeness} and Lemma \ref{lem:embedding lemma over substr}. 
By a standard compactness argument (cf. \cite[Corollary B.9.3]{ADH17}), it is enough to show that two complete $\mathcal L_{val}(K)$-types $q(x,y)$ and $r(x,y)$ are equal if their restrictions to a complete $\mathcal L_{ring}(\mathcal O_{v,n})$-type and a complete $\mathcal L_{oag}(vK)$-type are equal respectively.\\ 

Let $(K_1,v_1)$ and $(K_2,v_2)$ be elementary extensions of $(K,v)$, let $a_1,a_2$ be tuples in $\mathcal O_{v_1,n}$ and $\mathcal O_{v_2,n}$ respectively with $\tp_{\mathcal O_{v_1,n}}^{\mathcal L_{ring}}(a_1/\mathcal O_{v,n})=\tp_{\mathcal O_{v_2,n}}^{\mathcal L_{ring}}(a_2/\mathcal O_{v,n})$, and let $b_1,b_2$ be tuples in $v_1K_1$ and $v_2K_2$ respectively with $\tp_{v_1K_1}^{\mathcal L_{oag}}(b_1/vK)=\tp_{v_2K_2}^{\mathcal L_{oag}}(b_2/vK)$.
We will show that $\tp_{(K_1,v_1)}^{\mathcal L_{val}}(a_1,b_1/K)=\tp_{(K_2,v_2)}^{\mathcal L_{val}}(a_2,b_2/K)$.

By taking ultrapowers and elementary extensions properly, we may assume that $K$ and $K_1$ are $\aleph_1$-saturated and $K_2$ is $|K_1|^+$-saturated. 
Since $K_2$ is $|K_1|^+$-saturated and $\tp_{\mathcal O_{v_1,n}}^{\mathcal L_{ring}}(a_1/\mathcal O_{v,n})=\tp_{\mathcal O_{v_2,n}}^{\mathcal L_{ring}}(a_2/\mathcal O_{v,n})$, there is an elementary embedding $\varphi_{\mathcal O_n}:\mathcal O_{v_1,n}\rightarrow \mathcal O_{v_2,n}$ over $\mathcal O_{v,n}$ sending $a_1$ to $a_2$, and in a similar way, there is an elementary embedding $\varphi_{\Gamma}:v_1K_1\rightarrow v_2K_2$ over $vK$ sending $b_1$ to $b_2$.
By Lemma \ref{lem:embedding lemma over substr} with $K_0:=K$, $K:=K_1$, and $L:=K_2$, there is an embedding $\Phi:K_1\rightarrow K_2$ over $K$ which induces $\varphi_{\mathcal O_n}$ and $\varphi_{\Gamma}$. 
Also, by Theorem \ref{thm:relative model completeness}, the embedding $\Phi:K_1\rightarrow K_2$ is an elementary embedding over $K$. 
Thus, $\tp_{(K_1,v_1)}^{\mathcal L_{val}}(a_1,b_1/K)=\tp_{(K_2,v_2)}^{\mathcal L_{val}}(a_2,b_2/K)$.
\end{proof}

Next we consider the structure induced on the residue rings across all finitely ramified henselian valued fields of given ramification index, analogous to \cite[Proposition 6.5]{ADJ24}, which will turn out to be strongly related with special formulas. Using  Remark~\ref{rem:special formula under boolean combination},  
we rewrite Theorem \ref{thm:relative completeness} syntactically using (partial) types.

\begin{proposition}\label{prop:relative completeness in type}

Let $T_{p,e}$ be the $\mathcal L_{val}$-theory of all finitely ramified henselian valued fields of mixed characteristic $(0,p)$ having the initial ramification index $e$.

Let $k\ge 1 $ and $ n\ge \mr {max}\{k+e-1+ev_p(e), N_p(e) \}$. Let $q(x)$ be a complete $T_{p,e}$-type in a finite tuple $x$ of variables on the sort $\mathcal O_k$ and let $T_{\Gamma,q}$ be the set of all $\mathcal L_{oag}$-sentences in $q(x)$.
Then, the following hold modulo $T_{p,e}$:
\begin{enumerate}
	\item 
    For the set $\Sigma_q(x)$ of all special $\mathcal O_k$-formulas at level $n$ in the complete $T_{p,e}$-type $q(x)$, $$\Sigma_q(x)\cup T_{\Gamma,q}\models q(x).$$
	\item 
    For a complete $\mathcal L_{ring}$-type $q'(x')$ in a finite tuple $x'$ of variables on the sort $\mathcal O_n$, if $q(x)\models \exists x' \left(q'(x')\wedge \res^n_k(x')=x\right)$, then 
    $$\exists x' \left(q'(x')\wedge \res^n_k(x')=x\right)\cup T_{\Gamma,q}\models q(x).$$
\end{enumerate} 
 
\end{proposition}

\begin{proposition}\label{prop:equivlance between formula and special formula}
Let $T_{p,e}$ be the $\mathcal L_{val}$-theory of all finitely ramified henselian valued fields of mixed characteristic $(0,p)$ having the initial ramification index $e $. Let $k\ge 1 $  and 
$ n\ge \mr {max}\{k+e-1+ev_p(e), N_p(e) \}$.
Let $x$ be a tuple of variables on the sort $\mathcal O_k$. Then, any $\mathcal L_{val}$-formula $\theta(x)$ is equivalent modulo $T_{p,e}$ to a formula of the form 
$$\bigvee\left( \theta_{0}(x)\wedge \theta_{1}\right),$$
where $\theta_0(x)$ is a special $\mathcal O_k$-formula at level $n$ and $\theta_1$ is a $\mathcal L_{oag}$-sentence. 

Moreover, if $\theta(x)$ is existential, then it is equivalent modulo $T_{p,e}$ to an $\exists^+$-special $\mc O_k$-formula at level $n$.
\end{proposition}

\begin{proof}
We mimic the proof of \cite[Proposition 6.5]{ADJ24} with Proposition \ref{prop:relative completeness in type} (ii). By Proposition \ref{prop:relative completeness in type} (ii), each complete $\mathcal L_{val}$-$T_{p,e}$-type $q(x)$ is equivalent modulo $T_{p,e}$ to $\exists x' \left(q'(x')\wedge \res^n_k(x')=x\right)\cup T_{\Gamma,q}$ for a complete $\mathcal L_{ring}$-$T_{p,e}$-type $q'(x')$ in a tuple $x'$ of variables on the sort $\mathcal O_n$. 
So, by \cite[Corollary B.9.3]{ADH17}, $\theta(x)$  is equivalent modulo $T_{p,e}$ to a formula of the form 
$$\bigvee\left( \theta_0(x)\wedge \theta_1\right),$$ 
where $\theta_0(x)$ is a special $\mathcal O_k$-formula and $\theta_1$ is a $\mathcal L_{oag}$-sentence.
 
The existential assertion follows from Theorem~\ref{thm:relative existential completeness} and Remark~\ref{rem:special formula under boolean combination} by compactness, as in \cite[Corollary~5.9]{ADJ24}.
\end{proof}

\begin{definition} \label{def:existential-normal-form-depth}
For a prime $p$ and integers $e,k\ge1$, let $\lambda_{p}(e, k)$ be the least level at which every  $\mathcal L_{\mathrm{val}}$-formula in a tuple of variables on the sort $\mc O_k$ has a  normal form as in  \emph{Proposition~\ref{prop:equivlance between formula and special formula}}, uniformly over $T_{p,e}$.
\end{definition}

Note that Basarab’s invariant $\lambda(T)$ is the least length of a residue-ring that, together with the value-group theory, determines a fixed complete theory $T$ (cf. \cite[Definition~5.4]{LL21}).

\begin{theorem}  \label{thm:exact-normal-form-depth}
For every prime $p$ and integers $e,k\ge1$,
$$
 \lambda_{p}(e, k) =\max\{N_p(e),\,k+e-1+ev_p(e)\}.
$$ 
\end{theorem}

\begin{proof}
We simply write $C_k=k+e-1+ev_p(e)$.
By   Proposition~\ref{prop:equivlance between formula and special formula}, we obtain $
 \lambda_{p}(e, k) \le \max\{N_p(e),\,k+e-1+ev_p(e)\}.
$
For $e=1$, we have  $\lambda_p(e,k)= k$ since $n\ge k$. 

Now we assume $e>1$.
Consider  the following  formula
$$
\varphi_k(x):=\exists y\left(y^e +p= 0 \wedge\mr{res}_k(y)=x\right),
$$
where $y$ is a variable on the field sort.

First we  prove  $ \lambda_{p}(e, k)\ge N_p(e)$. 
 Propositions~\ref{prop:tame-pair} and  \ref{prop: wild nondivisible case}   give complete $\mathbb Z$-valued fields $K_1$ and $K_2$  such that  their residue rings   of   length  $<N_p(e)$ are isomorphic and $X^e + p$ has a root in $K_1$ only. 
Thus $\varphi_k$ defines a nonempty set in $K_1$ and the empty set in $K_2$. 
      If $k\le  \lambda_{p}(e, k) < N_p(e)$, we obtain a contradiction because an isomorphism of the residue rings of length $\lambda_p(e,k)$ preserves the projections of  definable sets.

Next we  prove  $ \lambda_{p}(e, k)\ge C_k$. Suppose $C_k >N_p(e)$. Let $K=\mathbb Q_p(\pi)$ with $\pi^e= -p$ and $v(\pi)=1$.
Put $f(X):=X^e +p$.
Since the roots of $f$ are $\pi\zeta$ with $\zeta^e=1$, an elementary calculation shows
$$
\Delta_v(f)=e-1+ev_p(e),\ \  \ \kappa_v(f)= 
\begin{cases}
1,&\hbox{if } \ p\nmid e,\\
1+e/(p-1),& \hbox{if } \ p\mid e.
\end{cases}
$$
Consequently, by Theorem~\ref{thm: lifting number}, we have
$
N_p(e)=\Delta_v(f)+\lfloor\kappa_v(f)\rfloor+1.
$ 
Then $k-1>\kappa_v(f)$. 
Set
$$
n=k+\Delta_v(f)-1\ \hbox{ and } \ \beta=\pi+\pi^{k-1}.
$$
For all roots  $\alpha$ $(\ne\pi)$ of $f$, we have $v(\beta-\alpha)=v(\pi-\alpha)<k-1$ by ultrametric inequality. 
Therefore we get
$$
v(f(\beta)) =k-1+\sum_{\alpha\ne\pi}v(\pi-\alpha) =k-1+\Delta_v(f)=n.
$$
In particular,   $\pi\mapsto\beta$ induces an automorphism  $\eta_n:\mc O_{v,n} \xrightarrow{\cong} \mc O_{v,n} $.

The   automorphism of $\mathcal O_{v,k}$, induced by $\eta_n$, sends $\mr{res}_k(\pi)$   to $\mr{res}_k(\beta)$.
Note that  $\mr{res}_k(\pi)$ satisfies $\varphi_k$, but $\mr{res}_k(\beta)$ does not since $v(\beta-\alpha)<k$ for every root $\alpha$ of $f$.
Every normal form as in Proposition~\ref{prop:equivlance between formula and special formula} at a level $m$ with $k\le m\le n$ is preserved by this automorphism, and hence $\varphi_k$ does not admit such a normal form at level less than $n+1$. 
It follows that $\lambda_p(e,k)\ge n+1=C_k$ and we get the desired formula 
$\lambda_p(e,k) =\mr{max} \{N_p(e), C_k \}$.
\end{proof}

Both normal forms modulo $T_{p,e}$ in Proposition~\ref{prop:equivlance between formula and special formula} are effectively computable at level $\lambda_p(e,k)$, uniformly in $p,e,k$ and the input formula.
Indeed, this level is computable, and the uniform computable axiomatizability of $T_{p,e}$ allows proof enumeration to find an equivalence of the required form.

\begin{remark}\label{rem:meaning of induced structure}
In \cite[Remark 6.3]{ADJ24}, it is known that the residue field sort and value group sort are stably embedded and orthogonal.
By Theorem \ref{thm:formulas on the higher residue ring}, we show that the multi-sorted structure of residue rings of all finite lengths and value group sort are also stably embedded and orthogonal. 

And, by \cite[Example 3.5]{LL21}, in the residue ring of higher length, there is  an $\mathcal L_{val}(\emptyset)$-definable but not $\mathcal L_{ring}(\emptyset)$-definable set.
\end{remark}

\subsection*{Acknowledgements}
The authors were supported by the National Research Foundation of Korea (NRF) grants funded by the Korea government (MSIT) (RS-2024-00451678). The second author was supported by the National Research Foundation of Korea (NRF) grants funded by the Korea government (MSIT) (RS-2021-NR061985) and by Basic Science Research Program through the National Research Foundation of Korea (NRF) funded by the Ministry of Education (RS-2025-25435653).
 

\subsection*{Statement on AI use}
The proof of the case $p-1 \mid e$   of Proposition~\ref{prop: wild nondivisible case} was obtained through conversations with ChatGPT (GPT-5.6 Sol, OpenAI).
In other parts of the paper, ChatGPT was used only for auxiliary tasks such as proofreading and reference checking.
The authors take full responsibility for all proofs in this paper.

\normalsize

\bibliographystyle{plain}
\bibliography{references}

\end{document}